\documentclass{amsart}
\pdfoutput=1
\usepackage{float}
\floatstyle{boxed}
\restylefloat{figure}
\usepackage{mathtools}
\usepackage{amssymb}
\usepackage{mathrsfs}
\usepackage[all]{xy}
\UseComputerModernTips
\CompileMatrices
\usepackage{enumitem}
\usepackage{graphicx} 
\usepackage{tikz}
\usepackage[colorlinks,allcolors=blue]{hyperref}
\usepackage{ifthen}
\usepackage{latexsym}
\allowdisplaybreaks
\numberwithin{equation}{section}
\newtheorem{theorem}[equation]{Theorem}
\newtheorem{lemma}[equation]{Lemma}
\newtheorem{proposition}[equation]{Proposition}
\newtheorem{corollary}[equation]{Corollary}
\newtheorem*{namedthm}{Theorem \namedthmname}
\newcounter{namedthm}

\makeatletter
\newenvironment{namedtheorem}[1]
  {\def\namedthmname{#1}%
   \refstepcounter{namedthm}%
   \namedthm\def\@currentlabel{#1}}
  {\endnamedthm}
\makeatother
\theoremstyle{definition}
\newtheorem{definition}[equation]{Definition}
\newtheorem{example}[equation]{Example}

\newtheorem{notation}[equation]{Notation}
\theoremstyle{remark}
\newtheorem{remark}[equation]{Remark}
\newtheorem{remarks}[equation]{Remarks}

\renewcommand{\phi}{\varphi}
\DeclareMathSymbol{\boxprod}{\mathbin}{AMSa}{"03} 
\DeclareMathSymbol{\mixprod}{\mathbin}{AMSa}{"4F} 

\newcommand{\dirsum}{\oplus}
\newcommand{\Dirsum}{\bigoplus}
\newcommand{\disjunion}{\sqcup}
\newcommand{\Disjunion}{\coprod}
\newcommand{\dual}{^\vee}

\newcommand{\hmtpc}{\simeq}
\newcommand{\homeo}{\approx}
\newcommand{\id}{\mathrm{id}}

\newcommand{\iso}{\cong}

\newcommand{\smsh}{\wedge}

\newcommand{\Susp}{\Sigma}
\newcommand{\susp}{\Sigma}
\newcommand{\tensor}{\otimes}

\newcommand{\C}{{\mathbb C}}

\newcommand{\PP}{\mathbb{P}}

\newcommand{\R}{{\mathbb R}}

\newcommand{\Z}{\mathbb{Z}}
\newcommand{\rels}[1]{\left\langle #1 \right\rangle}

\DeclareMathOperator{\Hom}{Hom}

\DeclareMathOperator{\im}{im}
\DeclareMathOperator{\grad}{grad}
\newcommand{\tensorS}{\tensor_{\HS}}
\newcommand{\HS}{{\mathbb H}}

\newcommand{\cw}{c_{\omega}}
\newcommand{\cxw}{c_{\chiw}}

\newcommand{\cwd}[1][]{\widehat{c}_\omega^{\ifthenelse{\equal{#1}{}}{}{{\:#1}}}}
\newcommand{\cxwd}[1][]{\widehat{c}_{\chiw}^{\ifthenelse{\equal{#1}{}}{}{{\:#1}}}}
\newcommand{\cl}{c_{\lambda}}
\newcommand{\cxl}{c_{\chi\lambda}}
\newcommand{\cld}{\widehat{c}_{\lambda}}
\newcommand{\cxld}{\widehat{c}_{\chi\lambda}}
\newcommand{\clo}{c_{\omega_1}}
\newcommand{\cxlo}{c_{\chi\omega_1}}
\newcommand{\clt}{c_{\omega_2}}
\newcommand{\cxlt}{c_{\chi\omega_2}}
\newcommand{\cltens}{c_{\omega_1\tensor\omega_2}}
\newcommand{\cxltens}{c_{\chi\omega_1\tensor\omega_2}}
\newcommand{\clod}[1][]{\widehat{c}_{\omega_1}^{\ifthenelse{\equal{#1}{}}{}{{\:#1}}}}
\newcommand{\cxlod}[1][]{\widehat{c}_{\chi\omega_1}^{\ifthenelse{\equal{#1}{}}{}{{\:#1}}}}
\newcommand{\cltd}[1][]{\widehat{c}_{\omega_2}^{\ifthenelse{\equal{#1}{}}{}{{\:#1}}}}
\newcommand{\cxltd}[1][]{\widehat{c}_{\chi\omega_2}^{\ifthenelse{\equal{#1}{}}{}{{\:#1}}}}
\newcommand{\cltensd}[1][]{\widehat{c}_{\omega_1\tensor\omega_2}^{\ifthenelse{\equal{#1}{}}{}{{\:#1}}}}
\newcommand{\cxltensd}[1][]{\widehat{c}_{\chi\omega_1\tensor\omega_2}^{\ifthenelse{\equal{#1}{}}{}{{\:#1}}}}
\newcommand{\cd}[1][]{\widehat{c}^{\ifthenelse{\equal{#1}{}}{}{{\:#1}}}}
\newcommand{\xd}[1][]{\widehat{x}^{\ifthenelse{\equal{#1}{}}{}{{\:#1}}}}
\newcommand{\yd}[1][]{\widehat{y}^{\ifthenelse{\equal{#1}{}}{}{{\:#1}}}}
\newcommand{\epsp}{\epsilon_\oplus} 
\newcommand{\Cpq}[2]{\C^{#1+#2\sigma}}

\newcommand{\Cq}[1]{\C^{#1\sigma}}
\newcommand{\Xpq}[2]{\PP(\Cpq{#1}{#2})}
\newcommand{\Xp}[1]{\PP(\C^{#1})}
\newcommand{\Xq}[1]{\PP(\Cq{#1})}
\newcommand{\chiw}{\chi\omega}
\newcommand{\ROev}{RO_{\mathrm{e}}}
\newcommand{\gr}{\Diamond}      
\newcommand{\ext}{\mathsf{\Lambda}}     
\newcommand{\GG}{{C_2}}
\begin{document}

\title{The $\GG$-equivariant ordinary cohomology of $BT^2$}

\author{Steven R. Costenoble}
\address{Steven R. Costenoble\\Department of Mathematics\\Hofstra University\\
  Hempstead, NY 11549, USA}
\email{Steven.R.Costenoble@Hofstra.edu}
\author{Thomas Hudson}
\address{Thomas Hudson, College of Transdisciplinary Studies, DGIST, 
Daegu, 42988, Republic of Korea}
\email{hudson@dgist.ac.kr}

%
%
\abstract
We calculate the ordinary $\GG$-cohomology of $BT^2$ with Burnside ring coefficients, 
using an extended grading that allows us to capture a more natural set of generators.
We discuss how this cohomology is related to those of $BT^1$ and $BU(2)$,
calculated previously,
both relationships being more complicated than in the nonequivariant case.
\endabstract
\maketitle
\tableofcontents

\section{Introduction}

This paper is a companion and followup to \cite{CH:bu2}. In that paper, we computed the 
$\GG$-equivariant ordinary cohomology of
$BU(2)$, the classifying space for $\GG$-equivariant complex 2-plane bundles,
using a cohomology theory with extended grading that allows us to capture an appropriate
set of characteristic classes, including Euler classes for all vector bundles.
Nonequivariantly, the cohomology of $BU(2)$ is often calculated by showing
that it is the symmetric part of the cohomology of $BT^2$.
Equivariantly, as we shall see, the relationship between these two cohomology rings
is not as straightforward.
Our first task is to compute the cohomology of $BT^2$ so that we can compare it to the
result of \cite{CH:bu2}. Calculations related to the one we do here have been done in the $RO(G)$-graded theory,
notably those by Lewis \cite{LewisCP}; more recent examples include work by Dugger
\cite{DuggerGrass}, Kronholm \cite{KronholmSerre}, Hogle \cite{Hogle} and Hazel \cite{HazelFundamental,HazelSurfaces}.

To set notations and for comparison purposes, recall the following nonequivariant results.
Thinking of $BT^2 = BT^1\times BT^1$, let $\omega_1$ and $\omega_2$ be the tautological
complex line bundles coming from the two factors of $BT^1$. We then have that
\begin{equation}\label{eqn:nonequiv bt2}
    H^*(BT^2_+;\Z) \iso H^*(BT^1_+;\Z)\tensor H^*(BT^1_+;\Z) \iso \Z[x_1,x_2]\,,
\end{equation}
by the K\"unneth theorem, where $x_1$ is the Euler class of $\omega_1$ and $x_2$ is the Euler class of $\omega_2$.
Moreover, we can relate this ring to the cohomology of $BU(2)$ by means of the well-known result
\begin{equation}\label{eqn:symmetry}
  H^*(BT^2_+;\Z)^{\Sigma_2}\iso
  H^*(BU(2)_+;\Z)  
\end{equation}
while the projective bundle formula gives
\begin{equation}\label{eqn:projective}
H^*(BT^2_+;\Z)\iso H^*(BU(2)_+;\Z)[x_1]/\rels{x_1^2-c_1x_1+c_2},
\end{equation}
where $c_1$ and $c_2$ are the Chern classes of the rank 2 tautological bundle $\omega$ over $BU(2)$. In particular, as an $H^*(BU(2)_+;\Z)$-module, $H^*(BT^2_+;\Z)$ is finitely generated and free of rank 2. One goal of this paper is to investigate the extent to which these relationships hold in the equivariant setting. Equations (\ref{eqn:nonequiv bt2}) and (\ref{eqn:symmetry}) will hold under some restrictions while, for (\ref{eqn:projective}), 
we will see that the equivariant cohomology of $BT^2$ is finitely generated over that of $BU(2)$ but no longer free.

We will assume the reader is familiar with the ordinary cohomology with extended grading we use,
from the summary given in \cite{CH:bu2}, the guide \cite{Beaudry:Guide}, or the exposition \cite{CostenobleWanerBook}.
As a model of equivariant $BT^2$ we shall take
\[
    BT^2 = BT^1\times BT^1 = \Xpq\infty\infty \times \Xpq\infty\infty,
\]
the product of two copies of an infinite projective space. 
Here, $\C$ is the trivial complex representation of $\GG$ and $\Cq{}$ is the sign representation.

The cohomology of $BT^1$
was computed in \cite{Co:InfinitePublished}.
This is a ring graded on
\[
    RO(\Pi BT^1) \iso \Z\{1,\sigma,\Omega_0,\Omega_1\} / \rels{\Omega_0 + \Omega_1 = 2\sigma - 2},
\]
the representation ring of the fundamental groupoid of $BT^1$,
where $1$ is the class of the constant representation $\R$ while $\sigma$ is the class
of the sign representation $\R^\sigma$.
For any complex vector bundle $\omega$, we write $\chi\omega = \omega\tensor\Cq{}$.
Then the main result of \cite{Co:InfinitePublished} was the following.
As in \cite{CH:bu2}, all cohomologies are taken with Burnside ring coefficients and
we write $\HS$ for $H_\GG^{RO(\GG)}(S^0)$, the $RO(\GG)$-graded cohomology of a point.

\begin{theorem}[{\cite[Theorem~11.3]{Co:InfinitePublished}}]\label{thm:bt1}
As a module, $H_\GG^{RO(\Pi BT^1)}(BT^1_+)$ is free over $\HS$, and as a commutative algebra
we have
\[
    H_\GG^{RO(\Pi BT^1)}(BT^1_+) \iso \HS[\zeta_0,\zeta_1,\cw,\cxw]/I
\]
where $I$ is the ideal generated by the relations
\begin{align*}
    \zeta_0\zeta_1 &= \xi \\
    \mathllap{\text{and}\quad}
    \zeta_1\cxw &= (1-\kappa)\zeta_0\cw + e^2.
\end{align*}
\qed
\end{theorem}

Here, $\cw$ is the Euler class of $\omega$, the tautological bundle over $BT^1$,
and $\cxw$ is the Euler class of $\chi\omega$. The elements $\zeta_0$ and $\zeta_1$
are related to the component structure of the fixed-set subspace $(BT^1)^\GG = BT^1 \disjunion BT^1$,
while $\xi$, $\kappa$, and $e$ are elements of $\HS$.

In \S\ref{sec:restricted grading}, shall show that, in a certain restricted set of gradings, the cohomology of $BT^2$
is the tensor product of two copies of the cohomology of $BT^1$;
see Corollary~\ref{cor:kunneth bt2}. 
However, this is not true if we use the full grading possible, which is the representation ring $RO(\Pi BT^2)$.
A calculation analagous to that of $RO(\Pi BT^1)$ shows that
\[
    RO(\Pi BT^2) \iso \Z\{1,\sigma,\Omega_{00},\Omega_{01},\Omega_{10},\Omega_{11}\} /
    \rels{\textstyle\sum\Omega_{ij} = 2\sigma - 2}.
\]
In stating our main result, we shall write $H_\GG^\gr(BT^2_+)$ for $H_\GG^{RO(\Pi BT^2)}(BT^2_+)$
and, in general, we will use the superscript $\gr$ to denote grading on $RO(\Pi BT^2)$.
We will prove the following over \S\S\ref{sec:structure} through~\ref{sec:main result}.

\begin{namedtheorem}{A}\label{thm:A}
As a module, $H_\GG^\gr(BT^2_+)$ is free over $\HS$, and as a commutative algebra we have
\[
    H_\GG^\gr(BT^2_+) \iso 
    \HS[\zeta_{00},\zeta_{01},\zeta_{10},\zeta_{11},\clo,\cxlo,\clt,\cxlt,\cltens,\cxltens] / I
\]
where $I$ is the ideal generated by the relations
\begin{align*}
    \zeta_{00}\zeta_{01}\zeta_{10}\zeta_{11} &= \xi \\
    \zeta_{10}\zeta_{11}\cxlo &= (1-\kappa) \zeta_{00}\zeta_{01}\clo + e^2 \\
    \zeta_{01}\zeta_{11}\cxlt &= (1-\kappa) \zeta_{00}\zeta_{10}\clt + e^2 \\
    \zeta_{00} \cltens &= \zeta_{10}\cxlo + \zeta_{01}\cxlt
		- e^{-2}\kappa \zeta_{01}\zeta_{10}\zeta_{11}\cxlo\cxlt \\
    \zeta_{11} \cltens &= \zeta_{01}\clo + \zeta_{10}\clt
		- e^{-2}\kappa \zeta_{00}\zeta_{01}\zeta_{10}\clo\clt \\
    \zeta_{01} \cxltens &= \zeta_{11}\cxlo + \zeta_{00}\clt
		- e^{-2}\kappa \zeta_{00}\zeta_{10}\zeta_{11}\cxlo\clt \\
    \zeta_{10} \cxltens &= \zeta_{00}\clo + \zeta_{11}\cxlt
		- e^{-2}\kappa \zeta_{00}\zeta_{01}\zeta_{11}\clo\cxlt \\
  \mathllap{\text{and}\quad}
    \cltens\cxltens &= \clo\cxlo + \clt\cxlt
		+ \tau(\iota^{-2})\zeta_{00}^2\zeta_{01}\zeta_{10}\clo\clt.
\end{align*}
\qed
\end{namedtheorem}

Here, $\omega_1$ and $\omega_2$ are the two tautological bundles over $BT^2$
and each element $c_L$ is the Euler class of the line bundle $L$.
As in Theorem~\ref{thm:bt1}, the $\zeta_{ij}$ correspond to the four
components of $(BT^2)^\GG$, while $\xi$, $\kappa$, $e$, $e^{-2}\kappa$, and $\tau(\iota^{-2})$
are elements of $\HS$.
Although this description looks considerably more complicated than (\ref{eqn:nonequiv bt2}),
it reduces to it: Each $\zeta_{ij}$ reduces to 1 nonequivariantly, $\clo$ and $\cxlo$ reduce to $x_1$,
$\clt$ and $\cxlt$ reduce to $x_2$, and $\cltens$ and $\cxltens$ reduce to $x_1 + x_2$.
The relations reduce to trivialities due to these reductions and the fact that
$\xi$ reduces to 1, $\tau(\iota^{-2})$ reduces to 2, and $\kappa$, $e$, and $e^{-2}\kappa$ all reduce to 0.

To help understand the structure shown in Theorem~\ref{thm:A}, in \S\ref{sec:bases} we give
examples of $\HS$-bases for several $RO(\GG)$ ``pages'' of the cohomology.
In \S\ref{sec:units}, we determine the units of the cohomology ring and use them
to calculate the Euler classes of dual bundles.

One of the most striking features of this calculation is the necessity of using the Euler class
$\cltens$ as one of the generators. Nonequivariantly, $e(\omega_1\tensor\omega_2) = e(\omega_1) + e(\omega_2)$,
which is the calculation underlying the fact that ordinary cohomology supports the additive formal group law.
Equivariantly, although there are relations involving $\cltens$, $\clo$, and $\clt$, we cannot
write $\cltens$ simply in terms of the other two.
Equivariant ordinary cohomology does not support a formal group law---we discuss this more
in \S\ref{sec:notes} along with some ideas that show that the set of relations above is
redundant if we take into account various self-maps of $BT^2$.

In \S\ref{sec:euler}, we show that, using $\cltens$, we can determine the Euler classes of
various other tensor products of tensor powers of line bundles.


We said that one of the reasons for doing this calculation was so that we could compare
the cohomologies of $BT^2$ and $BU(2)$ to see how (\ref{eqn:symmetry}) and (\ref{eqn:projective}) generalize.
We again have the map $s\colon BT^2\to BU(2)$ that classifies $\omega_1 \dirsum\omega_2$, so,
if $\omega$ denotes the tautological 2-plane bundle over $BU(2)$, $s^*\omega = \omega_1\dirsum\omega_2$.
One of the reasons we need to use the full $RO(\Pi BT^2)$ grading
is that $\omega_1\tensor\omega_2$ is the pullback of $\lambda = \ext^2\omega$,
and one of the generators of the cohomology of $BU(2)$ is $\cl$, the Euler class of $\lambda$.
Therefore, we need to use a set of gradings on $BT^2$ that allows us to have the element
$\cltens = s^*\cl$.
This means that the set of gradings used in \S\ref{sec:restricted grading},
in which the cohomology is the tensor product of the cohomologies
of the factors $BT^1$, is insufficient.

We go into detail about the relation between the cohomologies of $BT^2$ and $BU(2)$
in \S\S\ref{sec:projective} and~\ref{sec:pushforward}.
In general, the induced map $s^*$ from the cohomology of $BU(2)$ to the cohomology of $BT^2$ is
not even injective. However, it is injective when we restrict to \emph{even} gradings.
When we restrict to the even gradings, we get the familiar result that the cohomology of $BU(2)$ is exactly
the symmetric part of the cohomology of $BT^2$; this is Proposition~\ref{prop:symmetric}.
In \S\ref{sec:pushforward}, we explore the pushforward or transfer map from the cohomology
of $BT^2$ to the cohomology of $BU(2)$, which may be useful for further computations.

Nonequivariantly, we have the Whitney sum formula, which
is a powerful computational tool. Equivariantly, we don't have such a formula in general,
but in \S\ref{sec:waner}, we show that, for a set of classes suggested by Stefan Waner in an unpublished manuscript,
we do have a sum formula. Unfortunately, these classes
turn out to include only some of the most interesting ones.

\subsection*{Acknowledgements} 
Both authors would like to thank Sean Tilson for his help
in laying the foundations of this work. The first author thanks Hofstra University for released time to work on this project.
The second author was partially supported by the National Research Foundation of Korea (NRF) grant funded by the Korea government (MSIT) (No. RS-2024-00414849).

\section{A preliminary calculation}\label{sec:restricted grading}

Nonequivariantly, the K\"unneth theorem tells us that the cohomology of $BT^2$ is the tensor product of
two copies of the cohomology of $BT^1$. 
As we shall see, this is not true in the extended grading we use equivariantly,
but it is true in a restricted grading, and we describe this preliminary result before moving on
to the full grading.

\begin{definition}
Let $X$ and $Y$ be $\GG$-spaces and let $\pi_1\colon X\times Y\to X$ and $\pi_2\colon X\times Y\to Y$
be the projections. We define
\[
    SRO(\Pi X\times Y) = \im[\pi_1^* + \pi_2^*\colon RO(\Pi X)\dirsum RO(\Pi Y) \to RO(\Pi X\times Y)].
\]
\end{definition}

That is, $SRO(\Pi X\times Y)$ is the subgroup of $RO(\Pi X\times Y)$ consisting of those representations
that can be written as sums of representations of $\Pi X$ and $\Pi Y$, and
in general this is not all of $RO(\Pi X\times Y)$.
For example, the map $\pi_1^*\colon RO(\Pi BT^1)\to RO(\Pi BT^2)$ has
\[
    \pi_1^*(\Omega_0) = \Omega_{00} + \Omega_{01} \quad\text{and}\quad
    \pi_1^*(\Omega_1) = \Omega_{10} + \Omega_{11},
\]
while $\pi_2^*$ has
\[
    \pi_2^*(\Omega_0) = \Omega_{00} + \Omega_{10} \quad\text{and}\quad
    \pi_2^*(\Omega_1) = \Omega_{01} + \Omega_{11}.
\]
From this we can characterize the gradings in $SRO(\Pi BT^2)$ as
\[
    SRO(\Pi BT^2) = \left\{ a+b\sigma + {\textstyle \sum m_{ij}\Omega_{ij}} \in RO(\Pi BT^2)
        \mid m_{00} + m_{11} = m_{01} + m_{10} \right\},
\]
so that $RO(\Pi BT^2)/SRO(\Pi BT^2) \iso \Z$.
In particular, the representation $\omega_1\tensor\omega_2 = 2 + \Omega_{01} + \Omega_{10}$
is not in $SRO(\Pi BT^2)$.

\begin{lemma}
Assume that $X$ and $Y$ each have at least one $\GG$-fixed point. Then the addition map
\[
    RO(\Pi X)\dirsum_{RO(\GG)} RO(\Pi Y)\to SRO(\Pi X\times Y)
\]
is an isomorphism.
\end{lemma}

\begin{proof}
The map is an epimorphism by definition. Suppose that $\alpha_1, \alpha_2\in RO(\Pi X)$,
$\beta_1,\beta_2\in RO(\Pi Y)$, and $\alpha_1+\beta_1 = \alpha_2 + \beta_2 \in RO(\Pi X\times Y)$.
Then $\alpha_1 - \alpha_2 = \beta_1 - \beta_2$, so,
restricting to a fixed point in $Y$, we see that $\alpha_1 - \alpha_2$ must be constant on $\Pi X$,
that is, must be an element of $RO(\GG)$ (using that $X$ has a fixed point), and similarly for $\beta_1 - \beta_2$.
This shows the monomorphism needed to complete the lemma.
\end{proof}

This lemma can also be seen directly for $SRO(\Pi BT^2)$ from our characterization above.

\begin{proposition}\label{prop:kunneth}
Suppose that $H_\GG^{RO(\Pi X)}(X_+)$ is a free $\HS$-module and that $X$ and $Y$ each have a $\GG$-fixed point. Then
\[
    H_\GG^{SRO(\Pi X\times Y)}((X\times Y)_+) \iso H_\GG^{RO(\Pi X)}(X_+)\tensorS H_\GG^{RO(\Pi Y)}(Y_+).
\]
\end{proposition}

\begin{proof}
Fix $\gamma \in SRO(\Pi X\times Y)$. The lemma above shows that
$\gamma = \alpha+\beta$ for some $\alpha\in RO(\Pi X)$ and $\beta\in RO(\Pi Y)$,
determined up to an element of $RO(\GG)$.
Let $Z$ be an ex-$\GG$-space over $Y$.
We have the pairing
\[
    \mu\colon H_\GG^{\alpha+RO(\GG)}(X_+)\tensorS H_\GG^{\beta+RO(\GG)}(Z)
    \to H_\GG^{\alpha+\beta+RO(\GG)}(X_+\smsh Z).
\]
By the freeness of the cohomology of $X$, the left side is an $RO(\GG)$-graded cohomology theory on
spaces $Z$ over $Y$, while the right side is clearly one as well.
When $Z$ is a $\beta$-sphere over $Y$ (see \cite[Definition~3.1.1]{CostenobleWanerBook}),
$\mu$ is an isomorphism, hence it is one in general, in particular for $Z = Y_+$.
\end{proof}

Now we specialize to the spaces $BT^n$. With a slight abuse of notation, we define
\[
    SRO(\Pi BT^n) = RO(\Pi BT^1)\dirsum_{RO(\GG)} SRO(\Pi BT^{n-1})
\]
recursively, so elements of $SRO(\Pi BT^n)$ are ones that can be written as sums of representations
of the $n$ factors of $BT^1$.

\begin{corollary}\label{cor:kunneth bt2}
\[
    H_\GG^{SRO(\Pi BT^n)}(BT^n_+) \iso 
    \underbrace{H_\GG^{RO(\Pi BT^1)}(BT^1_+)\tensorS\cdots\tensorS H_\GG^{RO(\Pi BT^1)}(BT^1_+)}_{\text{$n$ factors}}.
\]
\end{corollary}

\begin{proof}
That the cohomology of $BT^1$ is free over $\HS$ was shown in \cite{Co:InfinitePublished}.
The result then follows by induction from Proposition~\ref{prop:kunneth}.
\end{proof}

The problem with this result is that there are interesting gradings outside of
$SRO(\Pi BT^n)$,
for example, the grading of the tensor product $\omega_1\tensor\omega_2$ over $BT^2$,
in which its Euler class lives, as mentioned above.
To calculate the cohomology of $BT^2$ in all gradings in $RO(\Pi BT^2)$ will require more work.

\section{The proposed structure}\label{sec:structure}

We shall eventually prove that the following is isomorphic to $ H_\GG^\gr(BT^2_+)$.

\begin{definition}\label{def:P}
Let $ P^\gr$ be the $RO(\Pi BT^2)$-graded commutative ring defined as the algebra over
$\HS$ generated by elements
\begin{itemize}
\item $\zeta_{00}$, $\zeta_{01}$, $\zeta_{10}$, $\zeta_{11}$, where $\grad \zeta_{ij} = \Omega_{ij}$,
\item $\clo$ and $\cxlo$, with $\grad \clo = 2 + \Omega_{10} + \Omega_{11}$ and
	$\grad \cxlo = 2 + \Omega_{00} + \Omega_{01}$,
\item $\clt$ and $\cxlt$, with $\grad \clt = 2 + \Omega_{01}+\Omega_{11}$ and
	$\grad \cxlt = 2 + \Omega_{00} + \Omega_{10}$,
\item $\cltens$ and $\cxltens$, with
	$\grad \cltens = 2 + \Omega_{01} + \Omega_{10}$ and
	$\grad \cxltens = 2 +\Omega_{00} + \Omega_{11}$,
\end{itemize}
subject to the relations
\begin{itemize}
\item $\prod_{ij} \zeta_{ij} = \xi$
\item $\zeta_{10}\zeta_{11}\cxlo = (1-\kappa) \zeta_{00}\zeta_{01}\clo + e^2$
\item $\zeta_{01}\zeta_{11}\cxlt = (1-\kappa) \zeta_{00}\zeta_{10}\clt + e^2$
\item $\zeta_{00} \cltens = \zeta_{10}\cxlo + \zeta_{01}\cxlt
		- e^{-2}\kappa \zeta_{01}\zeta_{10}\zeta_{11}\cxlo\cxlt$
\item $\zeta_{11} \cltens = \zeta_{01}\clo + \zeta_{10}\clt
		- e^{-2}\kappa \zeta_{00}\zeta_{01}\zeta_{10}\clo\clt$
\item $\zeta_{01} \cxltens = \zeta_{11}\cxlo + \zeta_{00}\clt
		- e^{-2}\kappa \zeta_{00}\zeta_{10}\zeta_{11}\cxlo\clt$
\item $\zeta_{10} \cxltens = \zeta_{00}\clo + \zeta_{11}\cxlt
		- e^{-2}\kappa \zeta_{00}\zeta_{01}\zeta_{11}\clo\cxlt$
\item $\cltens\cxltens =
		\clo\cxlo + \clt\cxlt
		+ \tau(\iota^{-2})\zeta_{00}^2\zeta_{01}\zeta_{10}\clo\clt$.
\end{itemize}
\end{definition}

\begin{notation}
For $\alpha = \sum_{ij} a_{ij}\Omega_{ij} \in RO(\Pi BT^2)$ with $a_{ij}\geq 0$, write
\[
 \zeta^{\alpha} = \prod_{i,j} \zeta_{ij}^{a_{ij}}.
\]
\end{notation}



For clarity, we write $BT^2 = U_1\times U_2$
to distinguish the two factors.
The ring $ H_\GG^{RO(\Pi U_1)}((U_1)_+)$ acts on $ P^\gr$
via the ring map defined on generators by
\begin{align*}
  \zeta_0 &\mapsto \zeta^{\chi\omega_1-2} \\
  \zeta_1 &\mapsto \zeta^{\omega_1-2} \\
  \cw &\mapsto \clo \\
  c_{\chi\omega} &\mapsto \cxlo.
\end{align*}
The relations in $ H_\GG^{RP(\Pi U_1)}((U_1)_+)$ from \cite{Co:InfinitePublished} map to relations above, so we do get a ring map.

A key result is the following.

\begin{proposition}\label{prop:freeness}
$ P^\gr$ is a free module over $ H_\GG^{RO(\Pi U_1)}((U_1)_+)$,
hence a free module over $ \HS$. As a module over $ H_\GG^{RO(\Pi U_1)}((U_1)_+)$,
it has a basis consisting of all the monomials in $\zeta_{ij}$, $\clt$, $\cxlt$,
$\cltens$, and $\cxltens$ that are {\em not} multiples of
any of
\begin{itemize}
\item $\zeta_{00}\zeta_{01}$,
\item $\zeta_{10}\zeta_{11}$,
\item $\zeta_{01}\zeta_{11} \cxlt$,
\item $\zeta_{00}^2 \clt$,
\item $\zeta_{10}^2 \clt$,
\item $\zeta_{01}^2 \cxlt$,
\item $\zeta_{11}^2 \cxlt$,
\item $\zeta_{00} \cltens$,
\item $\zeta_{11} \cltens$,
\item $\zeta_{01} \cxltens$,
\item $\zeta_{10} \cxltens$, or
\item $\cltens\cxltens$.
\end{itemize}
As a module over $\HS$ it has a basis consisting of all the monomials in
$\zeta_{ij}$, $\clo$, $\cxlo$, $\clt$, $\cxlt$,
$\cltens$, and $\cxltens$
that are {\em not} multiples of any of
\begin{itemize}
\item $\zeta_{00}\zeta_{01}\zeta_{10}\zeta_{11}$,
\item $\zeta_{10}\zeta_{11} \cxlo$,
\item $\zeta_{01}\zeta_{11} \cxlt$,
\item $\zeta_{00}^2\zeta_{01}^2 \clo$,
\item $\zeta_{00}^2 \clt$,
\item $\zeta_{10}^2 \clt$,
\item $\zeta_{01}^2 \cxlt$,
\item $\zeta_{11}^2 \cxlt$,
\item $\zeta_{00} \cltens$,
\item $\zeta_{11} \cltens$,
\item $\zeta_{01} \cxltens$,
\item $\zeta_{10} \cxltens$, or
\item $\cltens\cxltens$.
\end{itemize}
\end{proposition}

\begin{proof}
As an algebra over $ H_\GG^{RO(\Pi U_1)}((U_1)_+)$, $ P^\gr$ is generated by
$\zeta_{00}$, $\zeta_{01}$, $\zeta_{10}$, $\zeta_{11}$,
$\clt$, $\cxlt$,
$\cltens$, and $\cxltens$,
subject to the following relations, in which we think of $\zeta^{\omega_1-2}$, $\zeta^{\chi\omega_1-2}$,
$\clo$, and $\cxlo$ as elements of $ H_\GG^{RO(\Pi U_1)}((U_1)_+)$:
\begin{itemize}
\item $\zeta_{00}\zeta_{01} = \zeta^{\chi\omega_1-2}$
\item $\zeta_{10}\zeta_{11} = \zeta^{\omega_1-2}$
\item $\zeta_{01}\zeta_{11}\cxlt = (1-\kappa) \zeta_{00}\zeta_{10}\clt + e^2$
\item $\zeta_{00} \cltens 
		= \zeta_{10}\cxlo + (1-\kappa)(1-\epsilon_1)\zeta_{01}\cxlt$
\item $\zeta_{11} \cltens 
		= \zeta_{01}\clo + (1-\epsilon_1)\zeta_{10}\clt$
\item $\zeta_{01} \cxltens 
		= \zeta_{11}\cxlo + (1-\kappa)(1-\epsilon_1)\zeta_{00}\clt$
\item $\zeta_{10} \cxltens 
		= \zeta_{00}\clo + (1-\epsilon_1)\zeta_{11}\cxlt$
\item $\cltens\cxltens =
		\clo\cxlo + \clt\cxlt
		+ \tau(\iota^{-2})\zeta^{\chi\omega_1-2}\clo\zeta_{00}\zeta_{10}\clt$.
\end{itemize}
Here,
\[
    \epsilon_1 = e^{-2}\kappa \zeta_0\cw \in H_\GG^0((U_1)_+).
\]
Recall from \cite{Co:InfinitePublished} that $1-\kappa$ and $1-\epsilon_1$ 
are units in
$ H_\GG^{RO(\Pi U_1)}((U_1)_+)$, in fact their squares equal 1.

As in \cite{Co:InfinitePublished}, we will use Bergman's diamond lemma, Theorem~1.2 of \cite{Berg:diamondLemma}, as modified
for the commutative case by the comments in his \S10.3.
Let
\[
 E = \{\zeta_{00},\zeta_{01},\zeta_{10},\zeta_{11}, \cxlt, \clt,
  \cxltens, \cltens \}
\]
be the ordered set of generators and let $[E]$ denote the free commutative monoid on $E$, the set of monomials
in the generators.
Let $ Q \subset  P^\gr$ be the ideal generated by the relations above.
Then
\[
  P^\gr =  H_\GG^{RO(\Pi U_1)}((U_1)_+)[E]/ Q.
\]
We impose the following order on $[E]$. First we assign weights to each generator: The weights of
the $\zeta_{ij}$ are each $1$, the weights of $\clt$ and $\cxlt$ are each $2$,
and the weights of $\cltens$ and $\cxltens$ are each $4$.
A monomial 
\[
 (\prod_{ij}\zeta_{ij}^{a_{ij}})\clt^k \cxlt^\ell
  \cltens^m \cxltens^n
\]
then has weight $(\sum_{ij} a_{ij}) + 2(k+\ell) + 4(m+n)$.
We order monomials first by weight and then, for monomials of the same weight,
by reverse lexicographical order, so that a monomial with a higher power of $\cltens$
precedes one with a lower power, and if the powers are equal, the monomial with a higher power
of $\cxltens$ precedes the one with a lower power, and so on back through the order
specified for $E$.
This order obeys the multiplicative property that, if $A < B$ then $AC < BC$ for any monomial $C$.
It also obeys the descending chain condition (because there are only finitely many monomials of a given weight).

We also need to specify a {\em reduction system}, consisting of pairs $(W,f)$, where $W \in [E]$,
$f\in  H_\GG^{RO(\Pi U_1)}((U_1)_+)[E]$ is a polynomial with each of its monomials preceding $W$
in the order on $[E]$, and $W-f \in  Q$. Further, we need that $ Q$ is the ideal
generated by the collection of all the $W-f$. Here is the reduction system we shall use,
where each pair $(W,f)$ is written as $W\mapsto f$.
\begin{enumerate}[label=R\arabic*,series=reductions]
\item
\label{red:1}	$\zeta_{00}\zeta_{01} \mapsto \zeta^{\chi\omega_1-2}$

\item
\label{red:2}	$\zeta_{10}\zeta_{11} \mapsto \zeta^{\omega_1-2}$

\item
\label{red:3}	$\zeta_{01}\zeta_{11}\cxlt
				\mapsto (1-\kappa)\zeta_{00}\zeta_{10}\clt + e^2$

\item
\label{red:4}	$\zeta_{00}^2 \clt
				\mapsto \zeta^{\chi\omega_1-2}\cxltens
				- (1 - \kappa)(1 - \epsilon_1)\cxlo\zeta_{00}\zeta_{11}$

\item
\label{red:5}	$\zeta_{10}^2 \clt
				\mapsto \zeta^{\omega_1-2}\cltens
				- (1 - \epsilon_1)\clo\zeta_{01}\zeta_{10}$

\item
\label{red:6}	$\zeta_{01}^2 \cxlt
				\mapsto \zeta^{\chi\omega_1-2}\cltens
				- (1 - \kappa)(1 - \epsilon_1)\cxlo\zeta_{01}\zeta_{10}$

\item
\label{red:7}	$\zeta_{11}^2 \cxlt
				\mapsto \zeta^{\omega_1-2} \cxltens
				- (1 - \epsilon_1)\clo\zeta_{00}\zeta_{11}$

\item
\label{red:8}	$\zeta_{00}\cltens
				\mapsto \cxlo\zeta_{10} 
				+ (1-\kappa)(1-\epsilon_1)\zeta_{01}\cxlt$

\item
\label{red:9}	$\zeta_{11}\cltens
				\mapsto \clo\zeta_{01} 
				+ (1-\epsilon_1)\zeta_{10}\clt$

\item
\label{red:10}	$\zeta_{01}\cxltens
				\mapsto \cxlo\zeta_{11} 
				+ (1-\kappa)(1-\epsilon_1)\zeta_{00}\clt$

\item
\label{red:11}	$\zeta_{10}\cxltens
				\mapsto \clo\zeta_{00} 
				+ (1-\epsilon_1)\zeta_{11}\cxlt$

\item
\label{red:12}	$\cltens\cxltens
				\mapsto \clo\cxlo + \clt\cxlt
				+ \tau(\iota^{-2})\zeta^{\chi\omega_1-2}\clo\zeta_{00}\zeta_{10}\clt$
\end{enumerate}
The reader can check that the monomial on the right precedes the monomial on the left in each of 
these reductions.
The eight relations we specified appear as the first three and the last five 
reductions, so it is clear
that $W-f \in  Q$ for those eight. For \ref{red:4}, we have, working modulo $ Q$,
\begin{align*}
 \zeta_{00}^2 \clt
 	&\equiv (1-\kappa)(1-\epsilon_1)
			(\zeta^{\chi\omega_1-2}\cxltens - \cxlo\zeta_{00}\zeta_{11}) \\
	&= \zeta^{\chi\omega_1-2}\cxltens
			- (1-\kappa)(1-\epsilon_1)\cxlo\zeta_{00}\zeta_{11}
\end{align*}
where the congruence comes from multiplying the sixth relation by $\zeta_{00}$
and the equality follows from the facts that
$(1-\epsilon_1)\zeta_0 = (1-\kappa)\zeta_0$ and $(1-\kappa)^2 = 1$ 
in $ H_\GG^{RO(\Pi U_1)}((U_1)_+)$.
\ref{red:5}--\ref{red:7} are similar.

On the other hand, because all eight of our original relations appear as reductions, it is clear that $ Q$
is generated by the collection of elements $W-f$.

We think of each reduction as a way of rewriting a monomial $AW$ as $Af$, which we extend to polynomials.
What remains to be shown is that we can resolve all ambiguities in the reduction system. That is,
that if a monomial can be written both as $AW_1$ and $BW_2$, and we reduce to $Af_1$ and $Bf_2$,
respectively, then we can apply a sequence of further reductions to ultimately end at the same polynomial.
Thus, we need to take the reductions in pairs. 
This tedious check is relegated to the appendix and we assume it done.

Bergman's diamond lemma now applies to show that $ P^\gr$ is a free
module over $ H_\GG^{RO(\Pi U_1)}((U_1)_+)$ with the basis described in the proposition.
Because we know that $ H_\GG^{RO(\Pi U_1)}((U_1)_+)$
is a free module over $ \HS$, it follows that
$ P^\gr$ is also a free module over $ \HS$.
The basis given in the proposition follows from combining 
\cite[Theorems~10.3 and~11.3]{Co:InfinitePublished}, which
give a basis for
$ H_\GG^{RO(\Pi U_1)}((U_1)_+)$ over $ \HS$, with what we just
proved about the basis for $ P^\gr$ over $ H_\GG^{RO(\Pi U_1)}((U_1)_+)$.
\end{proof}

\section{The cohomology of the fixed set $(B T^2)^\GG$}\label{sec:fixed sets}

Write $(B T^2)^\GG = \Disjunion_{i,j\in\{0,1\}} B^{ij}$
where
\begin{align*}
 B^{00} &= \Xp\infty\times\Xp\infty \\
 B^{01} &= \Xp\infty\times\Xq\infty \\
 B^{10} &= \Xq\infty\times\Xp\infty \\
 B^{11} &= \Xq\infty\times\Xq\infty.
\end{align*}
Then we have the following calculation.

\begin{proposition}\label{prop:fixedpoints}
The cohomology of $(B T^2)^\GG$ is given by
\begin{multline*}
  H_\GG^\gr((B T^2)_+^\GG) \iso
  \Dirsum_{i,j \in \{0,1\}}  H_\GG^\gr(B^{ij}_+) \\
  \iso \Dirsum_{i,j} \HS[x_1,x_2,\zeta_{k\ell},\zeta_{k\ell}^{-1}
  						\mid k, \ell \in \{0,1\}, (k,\ell) \neq (i,j)]
\end{multline*}
where $\grad x_1 = \grad x_2 = 2$ and $\grad \zeta_{k\ell} = \Omega_{k\ell}$.
\end{proposition}

\begin{proof}
This is proved similarly to \cite[Proposition~6.4]{Co:InfinitePublished}.
The element $x_1$ comes from the nonequivariant Euler class of the tautological line bundle over the first factor
of $B^{ij}$ while $x_2$ comes from the second factor.
\end{proof}

One reason for looking at $H_\GG^\gr((B T^2)^\GG_+)$ is that, by \cite[Corollary~3.5]{CH:bu2},
the map
\[
    \eta\colon H_\GG^\gr(B T^2_+) \to H_\GG^\gr((B T^2)^\GG_+)
\]
induced by the inclusion $(B T^2)^\GG \to BT^2$ is a monomorphism in even gradings, meaning the following.

\begin{definition}\label{def:even grading}
Let $\ROev(\Pi BT^2) \subset RO(\Pi BT^2)$ be the subring defined by
\[
    \ROev(\Pi BT^2) = \left\{ a + b\sigma + {\textstyle\sum m_{ij}\Omega_{ij}} \mid \text{$a$ and $b$ even} \right\}.
\]
We call this the subgroup of \emph{even gradings.}
\end{definition}

The elements that we claim are generators all lie in even gradings, so we now
compute their images under $\eta$.

\begin{proposition}\label{prop:restrictions}
Write elements of $ H_\GG^\gr((B T^2)^\GG_+)$ as quadruples indexed by
$00$, $01$, $10$, and $11$ in that order. Then, under the restriction to fixed points map
\[
    \eta\colon H_\GG^\gr(B T^2_+) \to H_\GG^\gr((B T^2)^\GG_+),
\]
we have the following values of various elements:
\begin{align*}
 \eta(\zeta_{00}) &= (\ \xi\prod_{\mathclap{(i,j)\neq (0,0)}}\zeta_{ij}^{-1},\ \zeta_{00},\ \zeta_{00},\ \zeta_{00}) \\
 \eta(\zeta_{01}) &= (\zeta_{01},\ \xi\prod_{\mathclap{(i,j)\neq(0,1)}} \zeta_{ij}^{-1},\ \zeta_{01},\ \zeta_{01}) \\
 \eta(\zeta_{10}) &= (\zeta_{10},\ \zeta_{10},\ \xi\prod_{\mathclap{(i,j)\neq(1,0)}} \zeta_{ij}^{-1},\ \zeta_{10}) \\
 \eta(\zeta_{11}) &= (\zeta_{11},\ \zeta_{11},\ \zeta_{11},\ \xi\prod_{\mathclap{(i,j)\neq (1,1)}} \zeta_{ij}^{-1}) \\
 \eta(\clo) &= (x_1\zeta^{\omega_1-2},\ x_1\zeta^{\omega_1-2},\ 
 					 (e^2 + \xi x_1)\zeta^{\omega_1-2\sigma},\ (e^2 + \xi x_1)\zeta^{\omega_1-2\sigma}) \\
 \eta(\cxlo) &= ((e^2 + \xi x_1)\zeta^{\chi\omega_1-2\sigma},\ 
 					(e^2 + \xi x_1)\zeta^{\chi\omega_1-2\sigma},\ 
 					x_1\zeta^{\chi\omega_1-2},\ x_1\zeta^{\chi\omega_1-2}) \\
 \eta(\clt) &= (x_2\zeta^{\omega_2-2},\ (e^2+\xi x_2)\zeta^{\omega_2-2\sigma},\ 
 					x_2\zeta^{\omega_2-2},\ (e^2+\xi x_2)\zeta^{\omega_2-2\sigma}) \\
 \eta(\cxlt) &= ((e^2+\xi x_2)\zeta^{\chi\omega_2-2\sigma},\ x_2\zeta^{\chi\omega_2-2},\ 
 					(e^2+\xi x_2)\zeta^{\chi\omega_2-2\sigma},\ x_2\zeta^{\chi\omega_2-2}) \\
 \eta(\cltens) &= ((x_1+x_2)\zeta^{\omega_1\tensor\omega_2-2},\ 
				 (e^2+\xi (x_1+x_2))\zeta^{\omega_1\tensor\omega_2-2\sigma}, \\
		&\qquad	 (e^2+\xi(x_1+x_2))\zeta^{\omega_1\tensor\omega_2-2\sigma},
				 (x_1+x_2)\zeta^{\omega_1\tensor\omega_2-2}) \\
 \eta(\cxltens) &= ((e^2+\xi (x_1+x_2))\zeta^{\chi\omega_1\tensor\omega_2-2\sigma},\ 
				(x_1+x_2)\zeta^{\chi\omega_1\tensor\omega_2-2}, \\
		&\qquad	 (x_1+x_2)\zeta^{\chi\omega_1\tensor\omega_2-2},
				(e^2+\xi(x_1+x_2))\zeta^{\chi\omega_1\tensor\omega_2-2\sigma})
\end{align*}
\end{proposition}

\begin{proof}
The proof of these calculations is similar to the proofs of Proposition~6.5 and Corollary~8.2 of
\cite{Co:InfinitePublished}.
\end{proof}

With this, we can verify that the relations used in defining $P^\gr$ do actually
hold in $H_\GG^\gr(B T^2_+)$.

\begin{corollary}\label{cor:relations}
In $H_\GG^\gr(B T^2_+)$ we have the relations
\begin{align*}
    \prod_{i,j} \zeta_{ij} &= \xi \\
    \zeta_{10}\zeta_{11}\cxlo &= (1-\kappa) \zeta_{00}\zeta_{01}\clo + e^2 \\
    \zeta_{01}\zeta_{11}\cxlt &= (1-\kappa) \zeta_{00}\zeta_{10}\clt + e^2 \\
    \zeta_{00} \cltens &= \zeta_{10}\cxlo + \zeta_{01}\cxlt
		- e^{-2}\kappa \zeta_{01}\zeta_{10}\zeta_{11}\cxlo\cxlt \\
    \zeta_{11} \cltens &= \zeta_{01}\clo + \zeta_{10}\clt
		- e^{-2}\kappa \zeta_{00}\zeta_{01}\zeta_{10}\clo\clt \\
    \zeta_{01} \cxltens &= \zeta_{11}\cxlo + \zeta_{00}\clt
		- e^{-2}\kappa \zeta_{00}\zeta_{10}\zeta_{11}\cxlo\clt \\
    \zeta_{10} \cxltens &= \zeta_{00}\clo + \zeta_{11}\cxlt
		- e^{-2}\kappa \zeta_{00}\zeta_{01}\zeta_{11}\clo\cxlt \\
    \cltens\cxltens &= \clo\cxlo + \clt\cxlt
		+ \tau(\iota^{-2})\zeta_{00}^2\zeta_{01}\zeta_{10}\clo\clt.
\end{align*}
\end{corollary}

\begin{proof}
These equalities all involve elements in even gradings, where $\eta$ is injective, so can be checked one by one
by applying $\eta$ and using the calculations of the preceding proposition.
\end{proof}

\section{The cohomology of $B T^2$}\label{sec:main result}

It follows from Corollary~\ref{cor:relations} and the definition of $P^\gr$
that there is a ring map $P^\gr\to H_\GG^\gr(B T^2_+)$ taking the generators of $P^\gr$
to the elements of the same name in $H_\GG^\gr(B T^2_+)$.
Using \cite[Proposition~6.1]{CH:bu2}, we will prove that the image of a basis of $P^\gr$ is
a basis for $H_\GG^\gr(B T^2_+)$ by showing that the restrictions
of that basis to the nonequivariant cohomology of $BT^2$ and to the nonequivariant
cohomology of $(BT^2)^\GG$ are bases. This then implies that $P^\gr$ and $H_\GG^\gr(B T^2_+)$
are isomorphic rings.

\subsection{Restriction to nonequivariant cohomology}
We start with the map
\[
    \rho\colon H_\GG^\gr(B T^2_+) \to H_\rho^\gr(B T^2_+;\Z)
        = H^\gr(BT^2_+;\Z).
\]
We are regrading $H^\Z(BT^2_+;\Z)$ on $RO(\Pi BT^2)$
via the map $\rho\colon RO(\Pi BT^2)\to \Z$ given by
\[
    \rho(1) = 1 \qquad \rho(\sigma) = 1 \qquad \rho(\Omega_{ij}) = 0.
\]
(See \cite[\S5]{CH:bu2} for a discussion of regrading graded rings.)
The kernel of this map of gradings is the subgroup generated by the $\Omega_{ij}$ together with $\sigma-1$,
and there are corresponding
units in $H^\gr(BT^2_+;\Z)$ which we can identify with the elements $\rho(\zeta_{ij})$
and the element we have called $\iota$, respectively.
We write $\zeta_{ij} = \rho(\zeta_{ij})$ again.
Note that
\[
    \prod_{i,j} \zeta_{ij} = \rho(\xi) = \iota^2.
\]

Once more, write $\omega_1$ and $\omega_2$ for the two tautological line bundles
over $BT^2$, and let $x_1 = c_1(\omega_1)$ and $x_2 = c_1(\omega_2)$ be the corresponding
first Chern classes (that is, Euler classes).
Combining the familiar structure of the cohomology of $BT^2$ with the invertible elements introduced
by the regrading, we get the following calculation.
\begin{equation}\label{eqn:nonequivariant}
    H^\gr(BT^2;\Z) = \Z[x_1,x_2,\iota^{\pm 1},\zeta_{ij}^{\pm 1}]
    \Big/\rels{\prod\nolimits_{i,j}\zeta_{ij} = \iota^2}
\end{equation}
where
\[
    \grad x_1 = \grad x_2 = 2,\ \grad\iota = \sigma - 1, \text{ and } \grad\zeta_{ij} = \Omega_{ij}.
\]

The restrictions of various elements from the equivariant
cohomology of $B T^2$ are the following.
\begin{align*}
    \rho(\clo) &= \zeta_{10}\zeta_{11}x_1 = \zeta^{\omega_1-2}x_1 & \rho(\cxlo) &= \zeta^{\chi\omega_1-2}x_1 \\
    \rho(\clt) &= \zeta^{\omega_2-2}x_2 & \rho(\cxlt) &= \zeta^{\chi\omega_2-2}x_2 \\
    \rho(\cltens) &= \zeta^{\omega_1\tensor\omega_2-2}(x_1+x_2)
        & \rho(\cxltens) &= \zeta^{\chi\omega_1\tensor\omega_2-2}(x_1+x_2)
\end{align*}

Returning to $P^\gr$ for the moment, the main calculation we need is the following.
Recall that, when we regrade $H^\Z(S^0;\Z)$ on $RO(\GG)$ via $\rho\colon RO(\GG)\to\Z$, we get
\[
    H_\rho^{RO(\GG)}(S^0;\Z) \iso \Z[\iota^{\pm 1}].
\]

\begin{proposition}\label{prop:rho iso}
\[
    P^\gr\tensorS H_\rho^{RO(\GG)}(S^0;\Z) \iso 
    \Z[\clo,\clt,\iota^{\pm 1},\zeta_{ij}^{\pm 1}]
    \Big/\rels{\prod\nolimits_{i,j}\zeta_{ij} = \iota^2}
\]
\end{proposition}

\begin{proof}
We examine what happens with the relations defining $P^\gr$ when we change
the ground ring from $\HS$ to $H^{RO(\GG)}(S^0;\Z)$.
\begin{itemize}
\item $\prod_{i,j} \zeta_{ij} = \xi$ becomes
\[
    \prod_{ij}\zeta_{ij} = \iota^2
\]
\item $\zeta_{10}\zeta_{11}\cxlo = (1-\kappa) \zeta_{00}\zeta_{01}\clo + e^2$ becomes
\[
    \zeta_{10}\zeta_{11}\cxlo = \zeta_{00}\zeta_{01}\clo
\]
\item $\zeta_{01}\zeta_{11}\cxlt = (1-\kappa) \zeta_{00}\zeta_{10}\clt + e^2$ becomes
\[
    \zeta_{01}\zeta_{11}\cxlt = \zeta_{00}\zeta_{10}\clt
\]
\item $\zeta_{00} \cltens = \zeta_{10}\cxlo + \zeta_{01}\cxlt
		- e^{-2}\kappa \zeta_{01}\zeta_{10}\zeta_{11}\cxlo\cxlt$ becomes
\[
    \zeta_{00} \cltens = \zeta_{10}\cxlo + \zeta_{01}\cxlt
\]
\item $\zeta_{11} \cltens = \zeta_{01}\clo + \zeta_{10}\clt
		- e^{-2}\kappa \zeta_{00}\zeta_{10}\zeta_{01}\clo\clt$ becomes
\[
    \zeta_{11} \cltens = \zeta_{01}\clo + \zeta_{10}\clt
\]
\item $\zeta_{01} \cxltens = \zeta_{11}\cxlo + \zeta_{00}\clt
		- e^{-2}\kappa \zeta_{00}\zeta_{10}\zeta_{11}\cxlo\clt$ becomes
\[
    \zeta_{01} \cxltens = \zeta_{11}\cxlo + \zeta_{00}\clt
\]
\item $\zeta_{10} \cxltens = \zeta_{00}\clo + \zeta_{11}\cxlt
		- e^{-2}\kappa \zeta_{00}\zeta_{01}\zeta_{11}\clo\cxlt$ becomes
\[
    \zeta_{10} \cxltens = \zeta_{00}\clo + \zeta_{11}\cxlt
\]
\item $\cltens\cxltens =
		\clo\cxlo + \clt\cxlt
		+ \tau(\iota^{-2})\zeta_{00}^2\zeta_{01}\zeta_{10}\clo\clt$ becomes
\[
    \cltens\cxltens =
		\clo\cxlo + \clt\cxlt
		+ 2\zeta_{00}\zeta_{11}^{-1}\clo\clt
\]
\end{itemize}
The $\zeta_{ij}$ are invertible because $\prod \zeta_{ij} = \iota^2$ and $\iota$ is invertible,
so these relations now simply express
$\cxlo$, $\cxlt$, $\cltens$, and $\cxltens$ in terms of $\clo$ and $\clt$
(and many are redundant).
The result is then clear.
\end{proof}

This gives us part of what we need to verify our basis.

\begin{corollary}\label{cor:rho}
The composite
\[
    P^\gr\to H_\GG^\gr(B T^2_+) \xrightarrow{\rho}
        H_\rho^\gr(BT^2;\Z)
\]
takes a basis for $P^\gr$ over $\HS$
to a basis for $H_\rho^\gr(BT^2;\Z)$ over $H_\rho^{RO(\GG)}(S^0)$.
\end{corollary}

\begin{proof}
From the preceding proposition, calculation (\ref{eqn:nonequivariant}),
the fact that $\rho(\clo) = \zeta_{10}\zeta_{11}x_1$ and $\rho(\clt) = \zeta_{01}\zeta_{11}x_2$,
and the invertibility of the $\zeta_{ij}$, we get an isomorphism
\[
    P^\gr\tensorS H^{RO(\GG)}(S^0) \iso H^\gr(BT^2;\Z).
\]
The statement in the corollary now follows from the freeness of $P^\gr$.
\end{proof}

\subsection{Restriction to fixed points}\label{subsec:restrict fixed}
We now consider the map
\[
    \phi = (-)^\GG\colon H_\GG^\gr(B T^2_+)
    \to H_\phi^\gr((B T^2)^\GG_+;\Z)
\]
We expand a bit more on the grading on the right.
Recall that we have
\[
    (B T^2)^\GG = \Disjunion_{i,j} B^{ij}
\]
with each $B^{ij} \homeo BT^2$. We first think of $H^*((B T^2)^\GG_+;\Z)$ as
graded on $\Z^4$ via
\[
    H^a((B T^2)^\GG_+;\Z) = \Dirsum_{i,j} H^{a_{ij}}(BT^2_+;\Z)
\]
if $a = (a_{ij}) \in \Z^4$. We then have the fixed-point map
\[
    \phi\colon RO(\Pi BT^2) \to \Z^4
\]
given by
\[
    \textstyle \phi(a + b\sigma + \sum_{i,j} m_{ij}\Omega_{ij})
    = (a - 2m_{ij})_{i,j}.
\]
Thus, if $\alpha = a + b\sigma + \sum_{i,j} m_{ij}\Omega_{ij}$, then
\[
    H^\alpha_\phi((B T^2)^\GG_+;\Z) = \Dirsum_{i,j} H^{a - 2m_{ij}}(BT^2_+;\Z)
\]

If we now examine the $ij$th summand, we are regrading via a homomorphism
$RO(\Pi BT^2)\to \Z$ whose kernel has a basis given by $\sigma$ and
the $\Omega_{k\ell}$ with $(k,\ell)\neq (i,j)$.
There are corresponding invertible elements given by
$e = e^\GG$ and $\zeta_{ij} = \zeta_{ij}^\GG$.
This gives us the calculation
\begin{equation}\label{eqn:fixedpoints}
    H^\gr_\phi((B T^2)^\GG)_+;\Z) \iso
     \Dirsum_{i,j} \Z[x_1,x_2,e^{\pm 1},\zeta_{k\ell}^{\pm 1} \mid (k,\ell)\neq (i,j)]
\end{equation}
where
\[
    \grad x_1 = \grad x_2 = 2,\ \grad e = \sigma, \text{ and } \grad\zeta_{k\ell} = \Omega_{k\ell}.
\]

Recall from \cite[\S5]{CH:bu2} that, on regrading $H^\Z(S^0;\Z)$ via
$\phi\colon RO(\GG)\to \Z$, where $\phi(a+b\sigma) = a$, 
we get $H_\phi^{RO(\GG)}(S^0;\Z) \iso \Z[e^{\pm 1}]$.
This allows us to rewrite (\ref{eqn:fixedpoints}) as
\begin{equation}\label{eqn:fixedpoints2}
    H^\gr_\phi((B T^2)^\GG)_+;\Z) \iso
     \Dirsum_{i,j} H_\phi^{RO(\GG)}(S^0;\Z)[x_1,x_2,\zeta_{k\ell}^{\pm 1} \mid (k,\ell)\neq (i,j)]
\end{equation}

\begin{proposition}\label{prop:fixed point iso}
The map
\[
    P^\gr\tensorS H_\phi^{RO(\GG)}(S^0;\Z) \xrightarrow{\bar\phi} H_\phi^\gr((B T^2)^\GG)_+;\Z)
\]
induced by $\phi$ is an isomorphism of $H_\phi^{RO(\GG)}(S^0;\Z)$-modules.
\end{proposition}

\begin{proof}
We first record the value of $\bar\phi$ on various elements, where we write elements in the codomain
as 4-tuples, using (\ref{eqn:fixedpoints2}), in the order of indices $(00,01,10,11)$.
These values can be read off from the values of $\eta$ given in Proposition~\ref{prop:restrictions}.
\begin{align*}
    \bar\phi(\zeta_{00}) &= (0, \zeta_{00}, \zeta_{00}, \zeta_{00}) \\
    \bar\phi(\zeta_{01}) &= (\zeta_{01}, 0, \zeta_{01}, \zeta_{01}) \\
    \bar\phi(\zeta_{10}) &= (\zeta_{10}, \zeta_{10}, 0, \zeta_{10}) \\
    \bar\phi(\zeta_{11}) &= (\zeta_{11}, \zeta_{11}, \zeta_{11}, 0) \\
    \bar\phi(\clo) &= (\zeta^{\omega_1-2}x_1, \zeta^{\omega_1-2}x_1, 
                            e^2\zeta^{\omega_1-2\sigma}, e^2\zeta^{\omega_1-2\sigma}) \\
    \bar\phi(\cxlo) &= (e^2\zeta^{\chi\omega_1-2\sigma}, e^2\zeta^{\chi\omega_1-2\sigma},
                            \zeta^{\chi\omega_1-2}x_1, \zeta^{\chi\omega_1-2}x_1 ) \\
    \bar\phi(\clt) &= (\zeta^{\omega_2-2}x_2, e^2\zeta^{\omega_2-2\sigma},
                        \zeta^{\omega_2-2}x_2, e^2\zeta^{\omega_2-2\sigma}) \\
    \bar\phi(\cxlt) &= (e^2\zeta^{\chi\omega_2-2\sigma}, \zeta^{\chi\omega_2-2}x_2,
                        e^2\zeta^{\chi\omega_2-2\sigma}, \zeta^{\chi\omega_2-2}x_2 ) \\
    \bar\phi(\cltens) &= (\zeta^{\omega_1\tensor\omega_2-2}(x_1+x_2), e^2\zeta^{\omega_1\tensor\omega_2-2\sigma}, \\
            &\qquad\qquad e^2\zeta^{\omega_1\tensor\omega_2-2\sigma}, \zeta^{\omega_1\tensor\omega_2-2}(x_1+x_2) ) \\
    \bar\phi(\cxltens) &= (e^2\zeta^{\chi\omega_1\tensor\omega_2-2\sigma},
                            \zeta^{\chi\omega_1\tensor\omega_2-2}(x_1+x_2), \\
            &\qquad\qquad \zeta^{\chi\omega_1\tensor\omega_2-2}(x_1+x_2),
                            e^2\zeta^{\chi\omega_1\tensor\omega_2-2\sigma} )
\end{align*}
From these we get the following calculations.
\begin{align*}
    \bar\phi(e^{-2}\clo\cxlo) &= (x_1,x_1,x_1,x_1) \\
    \bar\phi(e^{-2}\clt\cxlt) &= (x_2,x_2,x_2,x_2) \\
    \bar\phi(e^{-4}\zeta^{\omega_1+\omega_2-4}\cxlo\cxlt) &= (1,0,0,0) \\
    \bar\phi(e^{-4}\zeta^{\omega_1+\chi\omega_2-4}\cxlo\clt) &= (0,1,0,0) \\
    \bar\phi(e^{-4}\zeta^{\chi\omega_1+\omega_2-4}\clo\cxlt) &= (0,0,1,0) \\
    \bar\phi(e^{-4}\zeta^{\chi\omega_1+\chi\omega_2-4}\clo\clt) &= (0,0,0,1) \\
    \bar\phi(e^{-4}\zeta^{\omega_2+\omega_1\tensor\omega_2-4-\Omega_{01}}\cxlt\cxltens) &= (\zeta_{01}^{-1},0,0,0) \\
    \bar\phi(e^{-4}\zeta^{\omega_1+\omega_1\tensor\omega_2-4-\Omega_{10}}\cxlo\cxltens) &= (\zeta_{10}^{-1},0,0,0) \\
    \bar\phi(e^{-4}\zeta^{\omega_1+\omega_2-4-\Omega_{11}}\cxlo\cxlt) &= (\zeta_{11}^{-1},0,0,0) \\
    \bar\phi(e^{-4}\zeta^{\chi\omega_2+\chi\omega_1\tensor\omega_2-4-\Omega_{00}}\clt\cltens) &= (0,\zeta_{00}^{-1},0,0) \\
\end{align*}
and similarly for the $\zeta_{ij}^{-1}$ appearing in the other possible positions.
We can now define an inverse $\psi$ to $\bar\phi$ as the unique algebra map with the following values:
\begin{align*}
    \psi(1,0,0,0) &= e^{-4}\zeta^{\omega_1+\omega_2-4}\cxlo\cxlt \\
    \psi(0,1,0,0) &= e^{-4}\zeta^{\omega_1+\chi\omega_2-4}\cxlo\clt \\
    \psi(0,0,1,0) &= e^{-4}\zeta^{\chi\omega_1+\omega_2-4}\clo\cxlt \\
    \psi(0,0,0,1) &= e^{-4}\zeta^{\chi\omega_1+\chi\omega_2-4}\clo\clt \\
    \psi(\zeta_{ij},0,0,0) &= \zeta_{ij}\psi(1,0,0,0), \text{ etc.} \\
    \psi(\zeta_{01}^{-1},0,0,0) &= e^{-4}\zeta^{\omega_2+\omega_1\tensor\omega_2-4-\Omega_{01}}\cxlt\cxltens, \text{ etc.} \\
    \psi(x_1,0,0,0) &= e^{-2}\clo\cxlo\psi(1,0,0,0), \text{ etc.}
\end{align*}
We can then check that $\psi$ is well-defined and is the inverse of $\bar\phi$, hence $\bar\phi$ is an isomorphism.
The checks are tedious but straightforward using the relations in $P^\gr\tensorS H^{RO(\GG)}_\phi(S^0;\Z)$.
\end{proof}

\begin{corollary}\label{cor:fixedpoints}
The composite
\[
    P^\gr \to H_\GG^\gr(B T^2_+)
    \xrightarrow{\phi} H_\phi^\gr((B T^2)^\GG_+;\Z)
\]
takes a basis for $P^\gr$ over $\HS$ to a basis for
$H_\phi^\gr((B T^2)^\GG_+;\Z)$ over $H_\phi^{RO(\GG)}(S^0;\Z)$.
\qed
\end{corollary}

\subsection{The cohomology of $B T^2$}
We can now prove Theorem~\ref{thm:A}, which we restate with more details as follows.

\begin{theorem}\label{thm:main}
$ H_\GG^\gr(B T^2_+)$ is generated as an algebra over $\HS$
by the elements $\zeta_{ij}$, $\clo$, $\cxlo$, $\clt$, $\cxlt$, $\cltens$, and $\cxltens$, modulo the relations
\begin{itemize}
\item $\prod_{i,j} \zeta_{ij} = \xi$
\item $\zeta_{10}\zeta_{11}\cxlo = (1-\kappa) \zeta_{00}\zeta_{01}\clo + e^2$
\item $\zeta_{01}\zeta_{11}\cxlt = (1-\kappa) \zeta_{00}\zeta_{10}\clt + e^2$
\item $\zeta_{00} \cltens = \zeta_{10}\cxlo + \zeta_{01}\cxlt
		- e^{-2}\kappa \zeta_{01}\zeta_{10}\zeta_{11}\cxlo\cxlt$
\item $\zeta_{11} \cltens = \zeta_{01}\clo + \zeta_{10}\clt
		- e^{-2}\kappa \zeta_{00}\zeta_{01}\zeta_{10}\clo\clt$
\item $\zeta_{01} \cxltens = \zeta_{11}\cxlo + \zeta_{00}\clt
		- e^{-2}\kappa \zeta_{00}\zeta_{10}\zeta_{11}\cxlo\clt$
\item $\zeta_{10} \cxltens = \zeta_{00}\clo + \zeta_{11}\cxlt
		- e^{-2}\kappa \zeta_{00}\zeta_{01}\zeta_{11}\clo\cxlt$
\item $\cltens\cxltens =
		\clo\cxlo + \clt\cxlt
		+ \tau(\iota^{-2})\zeta_{00}^2\zeta_{01}\zeta_{10}\clo\clt$.
\end{itemize}
It is a free module over $\HS$, with the basis given in
Proposition~\ref{prop:freeness}.
It is also a free module over $H_\GG^{RO(\Pi BT^1)}(B T^1_+)$, the action given by
projection to either factor, with the basis in the case of the first projection given in Proposition~\ref{prop:freeness}.
\end{theorem}

\begin{proof}
Corollaries~\ref{cor:rho} and~\ref{cor:fixedpoints}, 
together with \cite[Proposition~6.1]{CH:bu2},
imply that a basis of $ P^\gr$ over $\HS$ is taken to a basis
of $ H_\GG^\gr(B T^2_+)$,
hence the ring map $ P^\gr\to  H_\GG^\gr(B T^2_+)$
is an isomorphism.
The theorem then follows from the definition of $ P^\gr$
and Proposition~\ref{prop:freeness}.
\end{proof}

\section{Examples of bases}\label{sec:bases}

As we did in \cite{CH:bu2} for $BU(2)$, we now give the locations of some of the basis elements
of the cohomology of $BT^2$.
The specific basis we use is the $\HS$-basis given by Proposition~\ref{prop:freeness},
but the locations of the elements would not change if we used a different basis.
We look at ``$RO(\GG)$ pages,'' that is, 
at the groups graded by cosets $\alpha+RO(\GG)$
for fixed $\alpha\in RO(\Pi BT^2)$.

As the first example, we consider the $RO(\GG)$ grading itself, which we think of as the coset
of gradings $0 + RO(\GG)$. 
Below, we will also list the fixed sets of the basic elements, using the map
\[
    (-)^\GG\colon H_\GG^\gr(BT^2_+) \to H^{\Z^4}((BT^2)^\GG_+;\Z),
\]
grading the target on $\Z^4$ rather than regrading on $RO(\Pi BT^2)$.
(See \S\ref{subsec:restrict fixed} for more details.)

The beginning of the list of basic elements in the $RO(\GG)$ grading
is as follows, along with the grading of each and its fixed sets.
(This list can be derived from Proposition~\ref{prop:freeness} or from
Proposition~\ref{prop:kunneth}, which also applies in this range of gradings.)
\begin{align*}
    &x && \grad x && x^\GG \\
    &1 && 0 && (1,1,1,1) \\
    &\zeta_{00}\zeta_{01}\clo && 2\sigma && (0,0,1,1) \\
    &\zeta_{00}\zeta_{10}\clt && 2\sigma && (0,1,0,1) \\
    &\zeta_{00}^2\zeta_{01}\zeta_{10}\clo\clt && 4\sigma && (0,0,0,1) \\
    &\clo\cxlo && 2 + 2\sigma && (x_1,x_1,x_1,x_1) \\
    &\clt\cxlt && 2 + 2\sigma && (x_2,x_2,x_2,x_2) \\
    &\zeta_{00}\zeta_{01}\clo^2\cxlo && 2 + 4\sigma && (0,0,x_1,x_1) \\
    &\zeta_{00}\zeta_{10}\clo\cxlo\clt && 2 + 4\sigma && (0,x_1,0,x_1) \\
    &\zeta_{00}\zeta_{01}\clo\clt\cxlt && 2 + 4\sigma && (0,0,x_2,x_2) \\
    &\zeta_{00}\zeta_{10}\clt^2\cxlt && 2 + 4\sigma && (0,x_2,0,x_2) \\
    &\zeta_{00}^2\zeta_{01}\zeta_{10}\clo^2\cxlo\clt && 2 + 6\sigma && (0,0,0,x_1) \\
    &\zeta_{00}^2\zeta_{01}\zeta_{10}\clo\clt^2\cxlt && 2 + 6\sigma && (0,0,0,x_2)
\end{align*}
We draw the locations of these basis elements (and more) on a grid in which the location $a + b\sigma\in RO(\GG)$
is shown at point $(a,b)$. Because $a$ and $b$ are always even, the spacing of the grid lines is every 2, not 1.
The numbers in each circle represent the number of basis elements in that grading.
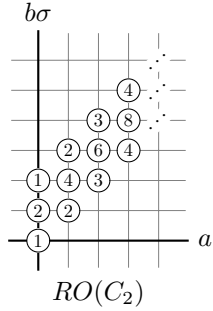
\begin{center}
\begin{tikzpicture}[scale=0.4] 
    \draw[step=1cm,gray,very thin] (-0.9,-0.9) grid (4.9,6.9);
    \draw[thick] (-1,0) -- (5,0);
    \draw[thick] (0,-1) -- (0,7);
    \node[right] at (5,0) {$a$};
    \node[above] at (0,7) {$b\sigma$};
    \node[below] at (2,-1) {$RO(\GG)$};

    \node[fill=white,scale=0.7] at (0,0) {1}; \draw (0,0) circle(0.38cm);
    \node[fill=white,scale=0.7] at (0,1) {2}; \draw (0,1) circle(0.38cm);
    \node[fill=white,scale=0.7] at (0,2) {1}; \draw (0,2) circle(0.38cm);

    \node[fill=white,scale=0.7] at (1,1) {2}; \draw (1,1) circle(0.38cm);
    \node[fill=white,scale=0.7] at (1,2) {4}; \draw (1,2) circle(0.38cm);
    \node[fill=white,scale=0.7] at (1,3) {2}; \draw (1,3) circle(0.38cm);

    \node[fill=white,scale=0.7] at (2,2) {3}; \draw (2,2) circle(0.38cm);
    \node[fill=white,scale=0.7] at (2,3) {6}; \draw (2,3) circle(0.38cm);
    \node[fill=white,scale=0.7] at (2,4) {3}; \draw (2,4) circle(0.38cm);

    \node[fill=white,scale=0.7] at (3,3) {4}; \draw (3,3) circle(0.38cm);
    \node[fill=white,scale=0.7] at (3,4) {8}; \draw (3,4) circle(0.38cm);
    \node[fill=white,scale=0.7] at (3,5) {4}; \draw (3,5) circle(0.38cm);

    \node[fill=white,scale=0.8,rotate=45] at (4,4) {\dots};
    \node[fill=white,scale=0.8,rotate=45] at (4,5) {\dots};
    \node[fill=white,scale=0.8,rotate=45] at (4,6) {\dots};

\end{tikzpicture}
\end{center}
Notice that, if we reduce each basis element using the non-regraded map
\[
    \rho\colon H_\GG^{RO(\GG)}(BT^2_+) \to H^\Z(BT^2_+;\Z),
\]
we get the familiar basis (shown in the same order as the elements above)
\[
    \{ 1, x_1, x_2, x_1x_2, x_1^2, x_2^2, x_1^3, x_1^2x_2, x_1x_2^2, x_2^3, x_1^3x_2, x_1x_2^3, \dots \}.
\]
The equivariant basis elements
that appear on the diagonal line of gradings $a+b\sigma$ with $a + b = n$
are the ones that reduce to nonequivariant basis elements
with total degree $n$.
We can illustrate the action of $\rho$ as follows, with the lower horizontal
line representing the nonequivariant cohomology of $BT^2$.
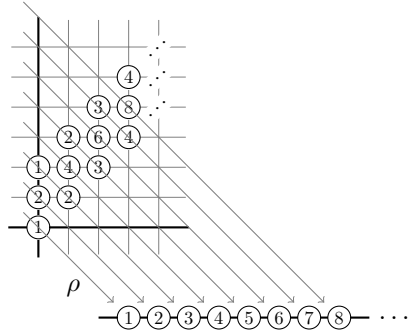
\begin{center}
\begin{tikzpicture}[scale=0.4] 

    \draw[step=1cm,gray,very thin] (-0.9,-0.9) grid (4.9,6.9);
    \draw[thick] (-1,0) -- (5,0);
    \draw[thick] (0,-1) -- (0,7);

    \node[fill=white,scale=0.7] at (0,0) {1}; \draw (0,0) circle(0.38cm);
    \node[fill=white,scale=0.7] at (0,1) {2}; \draw (0,1) circle(0.38cm);
    \node[fill=white,scale=0.7] at (0,2) {1}; \draw (0,2) circle(0.38cm);

    \node[fill=white,scale=0.7] at (1,1) {2}; \draw (1,1) circle(0.38cm);
    \node[fill=white,scale=0.7] at (1,2) {4}; \draw (1,2) circle(0.38cm);
    \node[fill=white,scale=0.7] at (1,3) {2}; \draw (1,3) circle(0.38cm);

    \node[fill=white,scale=0.7] at (2,2) {3}; \draw (2,2) circle(0.38cm);
    \node[fill=white,scale=0.7] at (2,3) {6}; \draw (2,3) circle(0.38cm);
    \node[fill=white,scale=0.7] at (2,4) {3}; \draw (2,4) circle(0.38cm);

    \node[fill=white,scale=0.7] at (3,3) {4}; \draw (3,3) circle(0.38cm);
    \node[fill=white,scale=0.7] at (3,4) {8}; \draw (3,4) circle(0.38cm);
    \node[fill=white,scale=0.7] at (3,5) {4}; \draw (3,5) circle(0.38cm);

    \node[fill=white,scale=0.8,rotate=45] at (4,4) {\dots};
    \node[fill=white,scale=0.8,rotate=45] at (4,5) {\dots};
    \node[fill=white,scale=0.8,rotate=45] at (4,6) {\dots};

    \draw[thick] (2,-3) -- (11,-3);
    \node[fill=white,scale=0.7] at (3,-3) {1}; \draw (3,-3) circle(0.38cm);
    \node[fill=white,scale=0.7] at (4,-3) {2}; \draw (4,-3) circle(0.38cm);
    \node[fill=white,scale=0.7] at (5,-3) {3}; \draw (5,-3) circle(0.38cm);
    \node[fill=white,scale=0.7] at (6,-3) {4}; \draw (6,-3) circle(0.38cm);
    \node[fill=white,scale=0.7] at (7,-3) {5}; \draw (7,-3) circle(0.38cm);
    \node[fill=white,scale=0.7] at (8,-3) {6}; \draw (8,-3) circle(0.38cm);
    \node[fill=white,scale=0.7] at (9,-3) {7}; \draw (9,-3) circle(0.38cm);
    \node[fill=white,scale=0.7] at (10,-3) {8}; \draw (10,-3) circle(0.38cm);

    \draw[->,gray] (-0.5,0.5) -- (2.5,-2.5);
    \draw[->,gray] (-0.5,1.5) -- (3.5,-2.5);
    \draw[->,gray] (-0.5,2.5) -- (4.5,-2.5);
    \draw[->,gray] (-0.5,3.5) -- (5.5,-2.5);
    \draw[->,gray] (-0.5,4.5) -- (6.5,-2.5);
    \draw[->,gray] (-0.5,5.5) -- (7.5,-2.5);
    \draw[->,gray] (-0.5,6.5) -- (8.5,-2.5);
    \draw[->,gray] (-0.5,7.5) -- (9.5,-2.5);
    \node[left] at (1.7,-2) {$\rho$};
    \node[right] at (11,-3) {\dots};

\end{tikzpicture}
\end{center}

The fixed sets listed above illustrate another feature of this basis. The basis elements in any vertical line
restrict to a basis for the fixed sets in a given grading, as shown in Corollary~\ref{cor:fixedpoints}.
For example, the eight basis elements in gradings of the form
$2 + b\sigma$ have fixed sets that give a basis for
\[
    H^2((BT^2)^\GG_+;\Z) \iso H^2(BT^2_+;\Z)\dirsum H^2(BT^2_+;\Z)\dirsum H^2(BT^2_+;\Z)\dirsum H^2(BT^2_+;\Z).
\]
A more familiar basis for this group would be
\begin{multline*}
    \{ (x_1,0,0,0), (x_2,0,0,0), (0,x_1,0,0), (0,x_2,0,0), \\
        (0,0,x_1,0), (0,0,x_2,0), (0,0,0,x_1), (0,0,0,x_2) \},
\end{multline*}
but it's not hard to see that the fixed sets shown in the list above form another basis.

We give a second example, from outside the range covered by Proposition~\ref{prop:kunneth}.
Consider the gradings $\omega_1\tensor\omega_2 + RO(\GG) = \Omega_{01} + \Omega_{10} + RO(\GG)$.
Here is the beginning of the list of basis elements in this coset.
\begin{align*}
    &x && \grad x = \Omega_{01} + \Omega_{10} + {} && x^\GG \\
    &\zeta_{01}\zeta_{10} && 0 && (1,0,0,1) \\
    &\zeta_{00}\zeta_{01}^2\zeta_{10}\clo && 2\sigma && (0,0,0,1) \\
    &\cltens && 2 && (x_1+x_2,1,1,x_1+x_2) \\
    &\zeta_{01}\zeta_{10}\clo\cxlo && 2 + 2\sigma && (x_1,0,0,x_1) \\
    &\zeta_{01}\zeta_{10}\clt\cxlt && 2 + 2\sigma && (x_2,0,0,x_2) \\
    &\zeta_{00}\zeta_{01}\clo\cltens && 2 + 2\sigma && (0,0,1,x_1+x_2) \\
    &\zeta_{00}\zeta_{01}^2\zeta_{10}\clo^2\cxlo && 2 + 4\sigma && (0,0,0,x_1) \\
    &\zeta_{00}\zeta_{01}^2\zeta_{10}\clo\clt\cxlt && 2 + 4\sigma && (0,0,0,x_2)
\end{align*}
The basis elements are arranged as follows:
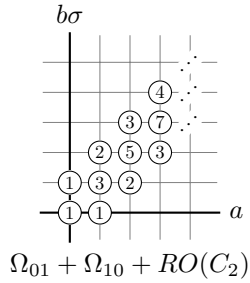
\begin{center}
\begin{tikzpicture}[scale=0.4] 
    \draw[step=1cm,gray,very thin] (-0.9,-0.9) grid (4.9,5.9);
    \draw[thick] (-1,0) -- (5,0);
    \draw[thick] (0,-1) -- (0,6);
    \node[right] at (5,0) {$a$};
    \node[above] at (0,6) {$b\sigma$};
    \node[below] at (2,-1) {$\Omega_{01} + \Omega_{10} + RO(\GG)$};

    \node[fill=white,scale=0.7] at (0,0) {1}; \draw (0,0) circle(0.38cm);
    \node[fill=white,scale=0.7] at (0,1) {1}; \draw (0,1) circle(0.38cm);

    \node[fill=white,scale=0.7] at (1,0) {1}; \draw (1,0) circle(0.38cm);
    \node[fill=white,scale=0.7] at (1,1) {3}; \draw (1,1) circle(0.38cm);
    \node[fill=white,scale=0.7] at (1,2) {2}; \draw (1,2) circle(0.38cm);

    \node[fill=white,scale=0.7] at (2,1) {2}; \draw (2,1) circle(0.38cm);
    \node[fill=white,scale=0.7] at (2,2) {5}; \draw (2,2) circle(0.38cm);
    \node[fill=white,scale=0.7] at (2,3) {3}; \draw (2,3) circle(0.38cm);

    \node[fill=white,scale=0.7] at (3,2) {3}; \draw (3,2) circle(0.38cm);
    \node[fill=white,scale=0.7] at (3,3) {7}; \draw (3,3) circle(0.38cm);
    \node[fill=white,scale=0.7] at (3,4) {4}; \draw (3,4) circle(0.38cm);

    \node[fill=white,scale=0.8,rotate=45] at (4,3) {\dots};
    \node[fill=white,scale=0.8,rotate=45] at (4,4) {\dots};
    \node[fill=white,scale=0.8,rotate=45] at (4,5) {\dots};

\end{tikzpicture}
\end{center}
The number of elements on the diagonal lines where $a+b$ is constant is the same as before, 
but the number on vertical lines
is different. This time, the fixed-set map has the form
\begin{multline*}
    (-)^\GG\colon H_\GG^{\Omega_{01} + \Omega_{10}+a+b\sigma}(BT^2_+) \to \\
    H^{a}(BT^2_+;\Z) \dirsum H^{a-2}(BT^2_+;\Z) \dirsum H^{a-2}(BT^2_+;\Z) \dirsum H^{a}(BT^2_+;\Z).
\end{multline*}
So, for example, the six basis elements in gradings of the form $\Omega_{01} + \Omega_{10}+2+b\sigma$ give a basis of
\[
    H^{2}(BT^2_+;\Z) \dirsum H^{0}(BT^2_+;\Z) \dirsum H^{0}(BT^2_+;\Z) \dirsum H^{2}(BT^2_+;\Z),
\]
for which a more familiar basis would be
\[
    \{ (x_1,0,0,0), (x_2,0,0,0), (0,1,0,0), (0,0,1,0), (0,0,0,x_1), (0,0,0,x_2) \}.
\]

Cosets of $RO(\GG)$ farther away from the 0 coset will give more exotic arrangements of basis elements.

\section{Units and dual elements}\label{sec:units}

We now determine the group of units of $H_\GG^\gr(BT^2_+)$.
We begin by introducing the following elements. 
(We used $\epsilon_1$ already in the proof of Proposition~\ref{prop:freeness}.)

\begin{definition}
Let
\begin{align*}
    \epsilon_1 &= e^{-2}\kappa\zeta_{00}\zeta_{01}\clo \\
    \epsilon_2 &= e^{-2}\kappa\zeta_{00}\zeta_{10}\clt \\
\intertext{and}
    \epsp &= e^{-4}\kappa\zeta_{00}^2\zeta_{01}\zeta_{10}\clo\clt.
\end{align*}
\end{definition}

(The notation for $\epsp$ indicates that it is based on $e(\omega_1\dirsum\omega_2) = \clo\clt$.)

From the basis for $H_\GG^{RO(\GG)}(BT^2_+)$ we gave in \S\ref{sec:bases}, 
we can see that $\{1,g,\epsilon_1,\epsilon_2,\epsp \}$ is a $\Z$-basis
for $H_\GG^0(BT^2_+)$.

\begin{proposition}\label{prop:units}
The elements $1-\kappa$, $1-\epsilon_1$, $1-\epsilon_2$, and $1-\epsp $
all square to $1$, hence are units.
There are exactly 32 units of $H_\GG^\gr(BT^2_+)$, given by
\[
    \pm(1-\kappa)^{a_0}(1-\epsilon_1)^{a_1}(1-\epsilon_2)^{a_2}(1-\epsp )^{a_3}
\]
where $a_i\in \{0,1\}$ for each $i$.
\end{proposition}

\begin{proof}
The argument given in \cite[Proposition~8.1]{CH:bu2} can be applied here to show that any
units of $H_\GG^\gr(BT^2_+)$ must lie in grading 0, where we have the $\Z$-basis given just
before this proposition.
To do calculations in that grading, we note the following products of the basis elements:
\begin{align*}
    g\epsilon_1 &= g\epsilon_2 = g\epsp  = 0 \\
    \epsilon_1^2 &= 2\epsilon_1 \\
    \epsilon_2^2 &= 2\epsilon_2 \\
    \epsp ^2 &= 2\epsp  \\
    \epsilon_1\epsilon_2 &= \epsilon_1\epsp  = \epsilon_2\epsp  = 2\epsp 
\end{align*}
The vanishing of the products with $g$ follows from the fact that $g\cdot e^{-2}\kappa = 0$.
The equations for $\epsilon_1^2$ and $\epsilon_2^2$ follow from the similar fact
about $e^{-2}\kappa\zeta_0\cw$ in the cohomology of $BT^1$ shown in 
\cite[Proposition~11.5]{Co:InfinitePublished}, or can be checked directly from
the relations in the cohomology of $BT^2$. 
For $\epsp ^2$, we calculate
\begin{align*}
    \epsp ^2
    &= (e^{-4}\kappa\zeta_{00}^2\zeta_{01}\zeta_{10}\clo\clt)^2 \\
    &= 2e^{-8}\kappa\zeta_{00}^2\zeta_{01}\zeta_{10}\clo\clt \cdot \zeta_{00}\zeta_{01}\clo \cdot \zeta_{00}\zeta_{10}\clt \\
    &= 2e^{-8}\kappa\zeta_{00}^2\zeta_{01}\zeta_{10}\clo\clt
        [(1-\kappa)\zeta_{10}\zeta_{11}\cxlo + e^2][(1-\kappa)\zeta_{01}\zeta_{11}\cxlt + e^2] \\
    &= 2e^{-4}\kappa\zeta_{00}^2\zeta_{01}\zeta_{10}\clo\clt \\
    &= 2\epsp ,
\end{align*}
using, as usual, the fact that $e^{-2}\kappa \cdot\xi = 0$. The remaining calculations are similar.

With the fact that $\kappa^2 = 2\kappa$, it now follows that each of $1-\kappa$, $1-\epsilon_1$, $1-\epsilon_2$,
and $1-\epsp $ squares to 1.

To see that there are no units other than the ones listed in the proposition, and that those are distinct,
we consider general elements
of $H_\GG^0(BT^2_+)$, which can be written uniquely in the form
\[
    u = a_1 + a_2g + a_3\epsilon_1 + a_4\epsilon_2 + a_5\epsp 
\]
for $a_i\in \Z$.
Multiplying two of these, using the relations above, we get
\begin{align*}
    (a_1 + a_2g + a_3\epsilon_1 + a_4\epsilon_2 + a_5\epsp )&(b_1 + b_2g + b_3\epsilon_1 + b_4\epsilon_2 + b_5\epsp ) \\
    &= a_1b_1 + (a_1b_2 + a_2b_1 + 2a_2b_2)g \\
    &\qquad + (a_1b_3 + a_3b_1 + 2a_3b_3)\epsilon_1 \\
    &\qquad + (a_1b_4 + a_4b_1 + 2a_4b_4)\epsilon_2 \\
    &\qquad + (a_1b_5 + 2a_3b_4 + 2a_3b_5 + 2a_4b_3 + 2a_4b_5 \\
    &\qquad\qquad+ a_5b_1 + 2a_5b_3 + 2a_5b_4 + 2a_5b_5)\epsp 
\end{align*}
Setting this equal to 1, we see that we must have $a_1b_1 = 1$, hence $a_1 = b_1 = \pm 1$.
Consider the case $a_1 = b_1 = 1$. Then we have
\[
    b_2 + a_2 + 2a_2b_2 = 0.
\]
As in the proof of \cite[Proposition~8.3]{CH:bu2}, the only integer solutions are $a_2 = b_2 = 0$ or $-1$.
Similarly, $a_3 = b_3 = 0$ or $-1$ and $a_4 = b_4 = 0$ or $-1$.

To determine $a_5$ and $b_5$ we look at various cases. We continue to assume that $a_1 = b_1 = 1$,
and note that $a_2$ and $b_2$ do not come into the final coefficient of the product above.
\begin{itemize}
    \item If $a_3 = b_3 = 0$ and $a_4 = b_4 = 0$: Then $b_5 + a_5 + 2a_5b_5 = 0$, from which
    we conclude that $a_5 = b_5 = 0$ or $-1$. The nontrivial cases give us the units
    \begin{align*}
        1-g &= -(1-\kappa) \\
        1-\epsp  \\
        \intertext{and}
        1-g-\epsp  &= (1-g)(1-\epsp ) = -(1-\kappa)(1-\epsp ).
    \end{align*}

    \item If $a_3 = b_3 = -1$ and $a_4 = b_4 = 0$: Then we must have
    \begin{align*}
        0 &= b_5 - 2b_5 + a_5 - 2a_5 + 2a_5b_5 \\
        &= -a_5 - b_5 + 2a_5b_5 \\
    \intertext{so}
        1 &= (1-2a_5)(1-2b_5).
    \end{align*}
    Since $a_5$ and $b_5$ are integers, this tells us that $1-2a_5 = 1-2b_5 = \pm 1$, hence
    $a_5 = b_5 = 0$ or $1$. This gives us the units
    \begin{align*}
        1 - \epsilon_1 \\
        1 - \epsilon_1 + \epsp  &= (1-\epsilon_1)(1-\epsp )
    \end{align*}
    and their products with $1-g$.

    \item If $a_3 = b_3 = 0$ and $a_4 = b_4 = -1$: This is similar to the preceding case
    and gives us the units $1-\epsilon_2$, $(1-\epsilon_2)(1-\epsp )$,
    and their products with $1-g$.

    \item If $a_3 = b_3 = -1$ and $a_4 = b_4 = -1$: Now we must have
    \begin{align*}
        0 &= b_5 + 2 - 2b_5 + 2 - 2b_5 - a_5 - 2a_5 - 2a_5 + 2a_5b_5 \\
        &= 4 - 3a_5 - 3b_5 + 2a_5b_5 \\
    \intertext{so}
        1 &= (3-2a_5)(3-2b_5).
    \end{align*}
    This gives us $3 - 2a_5 = 3-2b_5 = \pm 1$, hence $a_5 = b_5 = 1$ or $2$. This gives us the units
    \begin{align*}
        1 - \epsilon_1 - \epsilon_2 + \epsp  &= (1-\epsilon_1)(1-\epsilon_2)(1-\epsp ) \\
        1 - \epsilon_1 - \epsilon_2 + 2\epsp  &= (1-\epsilon_1)(1-\epsilon_2)
    \end{align*}
    and their products with $1-g$.
\end{itemize}
The cases in which $a_1 = b_1 = -1$ give us the negatives of the units found above.
Moreover, all of these cases are distinct because $\{1,g,\epsilon_1,\epsilon_2,\epsp \}$ is a basis.
Hence, the units are exactly the 32 listed in the statement of the proposition.
\end{proof}

\begin{remark}
As mentioned in \cite{CH:bu2}, we prefer to use $1-\kappa$ rather than $1-g$, so that each
of the units $1-\kappa$, $1-\epsilon_1$, $1-\epsilon_2$, and $1-\epsp $ restricts to 1 nonequivariantly.
\end{remark}

For a number of purposes, including studying pushforward maps as in \S\ref{sec:pushforward} below,
it is more natural to use generators based on the dual bundles $\omega_1\dual$ and $\omega_2\dual$.
There is a $\GG$-involution (up to homotopy)
\[
    \delta\colon BT^2\to BT^2
\]
classifying the pair of line bundles $(\omega_1\dual,\omega_2\dual)$, and we set
\begin{align*}
    \clod &= \delta^*\clo = e(\omega_1\dual) \\
    \cxlod &= \delta^*\cxlo = e(\chi\omega_1\dual) \\
    \cltd &= \delta^*\clt = e(\omega_2\dual) \\
    \cxltd &= \delta^*\cxlt = e(\chi\omega_2\dual) \\
    \cltensd &= \delta^*\cltens = e((\omega_1\tensor\omega_2)\dual) \\
    \cxltensd &= \delta^*\cxltens = e(\chi(\omega_1\tensor\omega_2)\dual).
\end{align*}
Since $\delta^*$ is an algebra isomorphism, these elements can be used as generators,
satisfying the relations listed in Theorem~\ref{thm:main}, \textit{mutatis mutandis.}
But they must also be able to be written in terms of the original generators, and we now work out those formulas.

\begin{proposition}\label{prop:duals}
The dual classes can be written in terms of the usual generators as
\begin{align*}
    \clod &= -(1-\epsilon_1)\clo \\
    \cxlod &= -(1-\kappa)(1-\epsilon_1)\cxlo \\
    \cltd &= -(1-\epsilon_2)\clt \\
    \cxltd &= -(1-\kappa)(1-\epsilon_2)\cxlt \\
    \cltensd &= -(1-\epsilon_1)(1-\epsilon_2)\cltens \\
    \cxltensd &= -(1-\kappa)(1-\epsilon_1)(1-\epsilon_2)\cxltens.
\end{align*}
\end{proposition}

\begin{proof}
The first four formulas follow from the similar formulas in the cohomology of $BT^1$,
shown in \cite[Proposition~11.6]{Co:InfinitePublished}.
That result also implies the first equality in the following.
\begin{align*}
    \cltensd &= -(1 - e^{-2}\kappa\zeta_{00}\zeta_{11}\cltens)\cltens \\
    &= -[1 - e^{-2}\kappa\zeta_{00}(\zeta_{01}\clo + \zeta_{10}\clt - e^{-2}\kappa\zeta_{00}\zeta_{01}\zeta_{10}\clo\clt)]\cltens \\
    &= -(1 - \epsilon_1 - \epsilon_2 + 2\epsp )\cltens \\
    &= -(1-\epsilon_1)(1-\epsilon_2)\cltens.
\end{align*}
The formula for $\cxltensd$ follows similarly.
\end{proof}

\section{Notes on the relations}\label{sec:notes}

If we take into account several self-maps of $BT^2$, there is a sense in which the only three
relations we need from Theorem~\ref{thm:main} are
\begin{align}
    \prod\nolimits_{i,j} \zeta_{ij} &= \xi \\
    \zeta_{11} \cltens &= \zeta_{01}\clo + \zeta_{10}\clt
		- e^{-2}\kappa \zeta_{00}\zeta_{01}\zeta_{10}\clo\clt \label{eqn:main rel} \\
    \cltens\cxltens &=
		\clo\cxlo + \clt\cxlt
		+ \tau(\iota^{-2})\zeta_{00}^2\zeta_{01}\zeta_{10}\clo\clt. \label{eqn:product rel}
\end{align}
For example, there is the self map $\chi_1\colon BT^2\to BT^2$ that applies $\chi$ to the first
factor of $BT^1$. When we apply $\chi_1^*$ to (\ref{eqn:main rel}), we get the relation
\[
    \zeta_{01}\cxltens = \zeta_{11}\cxlo + \zeta_{00}\clt - e^{-2}\kappa \zeta_{00}\zeta_{10}\zeta_{11}\cxlo\clt.
\]
The relations involving $\zeta_{10}\cxltens$ and $\zeta_{00}\cltens$ can be gotten similarly.

The relation for $\zeta_{10}\zeta_{11}\cxlo$ really comes from the similar relation
in the cohomology of $BT^1$ shown in \cite{Co:InfinitePublished}, that
\begin{equation}\label{eqn:bt1 fundamental}
    \zeta_1\cxw = (1-\kappa)\zeta_0\cw + e^2,
\end{equation}
but that relation can be derived from (\ref{eqn:main rel}) as well. There are several ways to do this,
perhaps the simplest being the following. Consider the map
$t\colon BT^1\to BT^2$ classifying the pair of line bundles $(\omega,\Cq{})$,
so that $t^*(\omega_1\tensor\omega_2) = \chi\omega$.
In cohomology, we have
\begin{align*}
    t^*(\zeta_{00}) &= 1  & t^*(\zeta_{01}) &= \zeta_0 \\
    t^*(\zeta_{10}) &= 1 & t^*(\zeta_{11}) &= \zeta_1 \\
    t^*(\clo) &= \cw & t^*(\cxlo) &= \cxw \\
    t^*(\clt) &= e^2 & t^*(\cxlt) &= 0 \\
    t^*(\cltens) &= \cxw & t^*(\cxltens) &= \cw
\end{align*}
If we apply $t^*$ to (\ref{eqn:main rel}), we get
\begin{align*}
    \zeta_1\cxw &= t^*(\zeta_{11}\cltens) \\
    &= t^*(\zeta_{01}\clo + \zeta_{10}\clt
		- e^{-2}\kappa \zeta_{00}\zeta_{01}\zeta_{10}\clo\clt) \\
    &= \zeta_0\cw + e^2 - e^{-2}\kappa \zeta_0\cw\cdot e^2 \\
    &= (1-\kappa)\zeta_0\cw + e^2,
\end{align*}
which is (\ref{eqn:bt1 fundamental}).
We can then get the relation for $\zeta_{10}\zeta_{11}\cxlo$ given in Theorem~\ref{thm:main}
by pulling back (\ref{eqn:bt1 fundamental}) along the first projection $BT^2\to BT^1$,
and the relation for $\zeta_{01}\zeta_{11}\cxlt$ by pulling back along the second projection.

\begin{remark}\label{rem:global approach}
There is something deeper going on here that may turn out to be useful in describing the cohomologies of the $BT^n$
in general. We sketch the idea here, as describing it in detail would take us too far afield,
and not all of the details have been worked out, yet.
We could consider a category with objects $* = BT^0$, $BT^1$, and $BT^2$, and certain maps
between them including those used above, classifying tuples of tensor products of
tensor powers of tautological bundles (including $\C$ and $\Cq{}$).
We can then consider the cohomologies of $*$, $BT^1$, and $BT^2$ as giving a contravariant
functor on this category, into a category of graded rings with varying grading groups.
This functor can then be described as generated by
the elements $\zeta_i$ in the cohomology of $BT^1$, $\zeta_{ij}$ in the cohomology of $BT^2$,
and a single element $\cw$ in the cohomology of $BT^1$,
as we can obtain $\cxw$, $\clo$, $\cxlo$, and so on as the images of $\cw$ under 
the classifying maps for $\chi\omega$ and so on, which are maps in this category.
The relations required can then be reduced to ones describing the relations of the $\zeta$s
and the relations (\ref{eqn:main rel}) and (\ref{eqn:product rel}) 
(and one more, describing $e(\omega^{\tensor 2})$; see Remark~\ref{rem:missing relation}).

To extend this to a functor encompassing the cohomologies of all of the $BT^n$ appears to be not
as straightforward as we would hope. Preliminary calculations suggest that a new generator in the cohomology
of $BT^4$ is required, and likely there are further such generators required for larger $n$.
\end{remark}

One more point to make about the relations. (\ref{eqn:main rel}) is as close as we come to being
able to write $\cltens = e(\omega_1\tensor\omega_2)$ in terms of $\clo = e(\omega_1)$ and $\clt = e(\omega_2)$;
because $\zeta_{11}$ is not invertible, we cannot solve for $\cltens$.
The upshot of this is that we do not have a formal group law, unlike nonequivariant ordinary cohomology,
which supports the additive formal group law.

There is active research on equivariant formal group laws, introduced in \cite{CGK:formalGroupLaws},
that are present in \emph{complex oriented} theories,
a subclass of the class of \emph{complex stable} theories.
The latter are required to have suspension isomorphisms
of the form $E_G^\alpha(X) \iso E_G^{\alpha+|V|}(\Susp^V X)$ for all complex representations $V$ of $G$,
where $|V|$ is the integer real dimension of $V$,
so in a sense they do not distinguish between different complex representations of the group.
(Likewise, complex oriented theories in this sense have Thom classes for complex vector bundles
that live in integer grading, whereas ordinary cohomology has Thom classes that live in gradings
in the representation ring of the fundamental groupoid of the basespace, generally not in integer gradings.)

We could turn ordinary cohomology $H_\GG^{RO(\GG)}$ into a complex stable theory by inverting the element $\xi$,
but this has exactly the effect of turning Bredon cohomology into Borel cohomology, which is indeed
complex oriented and does carry an additive formal group law.
We can see that explicitly, because inverting $\xi$ inverts all the $\zeta_{ij}$ as well,
allowing us to express $\cltens$ in terms of $\clo$ and $\clt$.
In addition, when we invert $\xi$, $e^{-2}\kappa$ becomes 0,
so the third term in (\ref{eqn:main rel}) goes away, leaving a simple additive relation.
On the other hand, as discussed in \cite{CHTAlgebraic}, there are distinct advantages
to Bredon cohomology over Borel cohomology,
so we can live without a formal group law.

\section{Euler classes of tensor products}\label{sec:euler}

As mentioned at the end of the preceding section, we do not have a formal group law
and cannot, in general, express the Euler class of a tensor product in terms
of the Euler classes of the factors.
However, there are some interesting special cases that we can calculate.

Adapting the notation from algebraic geometry, for $n\in \Z$ we write
\begin{align*}
    O(n) &= (\omega(1)\dual)^{\tensor n} \\
    \mathllap{\text{and }} \chi O(n) &= O(n)\tensor\Cq{}
\end{align*}
for the indicated line bundles over $BT^1$. 
In \cite[Proposition~6.5]{CHTFiniteProjSpace} we calculated the Euler classes of these line bundles,
one in particular being
\begin{equation}\label{eqn:Q}
    Q = e(O(2)) = \cwd(\tau(\iota^{-2})\zeta_0 + e^{-2}\kappa\cxwd).
\end{equation}
We will use this calculation to derive a number of other Euler classes.

For $m,n\in \Z$, write
\begin{align*}
    O(m,n) &= (\omega_1\dual)^{\tensor m}\tensor(\omega_2\dual)^{\tensor n} \\
    \mathllap{\text{and }} \chi O(m,n) &= O(m,n)\tensor\Cq{}
\end{align*}
for the indicated line bundles over $BT^2$.
We will compute the Euler classes of these bundles using the following general result
(which we will use again in \S\ref{sec:pushforward}).

\begin{proposition}\label{prop:tensor w fixed}
Let $X$ be a $\GG$-space with complex line bundles $\omega$ and $\mu$, where
the fibers of $\mu$ all have trivial $\GG$-action. 
Write
\[
    X^\GG = X^0 \disjunion X^1
\]
where the fibers of $\omega$ over $X^0$ have trivial $\GG$-action while the
fibers over $X^1$ are copies of $\Cq{}$. 
(So $X^0$ and $X^1$ are, in general, unions of components of $X^\GG$.)
Let $\zeta_0$ and $\zeta_1$ be the
pullbacks of $\zeta_0$ and $\zeta_1$ along the map $X\to BU(1)$ classifying $\omega$.
Then
\[
    e(\omega\tensor\mu) = e(\omega) + \zeta_1 e(\mu).
\]
\end{proposition}

\begin{proof}
Let $p\colon X\to B T^2$ classify the pair $(\omega,\mu)$. Then $X^0$ maps to $T^{00}$
and $X^1$ maps to $T^{10}$, so that
\begin{align*}
    p^*(\zeta_{00}) &= \zeta_0  &  p^*(\zeta_{01}) &= 1 \\
    p^*(\zeta_{10}) &= \zeta_1  &  p^*(\zeta_{11}) &= 1 \\
    p^*(\clo) &= e(\omega)  &  p^*(\clt) &= e(\mu) \\
    p^*(\cltens) &= e(\omega\tensor\mu)  &  p^*(\cxltens) &= e(\chi\omega\tensor\mu).
\end{align*}
We then have
\begin{align*}
    e(\omega\tensor\mu) &= p^*(\zeta_{11}\cltens) \\
        &= p^*(\zeta_{01}\clo + \zeta_{10}\clt - e^{-2}\kappa \zeta_{00}\zeta_{01}\zeta_{10}\clo\clt) \\
        &= e(\omega) + \zeta_1 e(\mu),
\end{align*}
with no third term because $e^{-2}\kappa\cdot\xi = 0$.
\end{proof}

We now compute the Euler classes of the line bundles $O(m,n)$ and $\chi O(m,n)$.
With the algebro-geometric convention, it's most natural to do this in terms of the dual generators.

\begin{proposition}
Let
\begin{align*}
    Q_1 &= \clod(\tau(\iota^{-2})\zeta_{00}\zeta_{01} + e^{-2}\kappa\cxlod) \\
    \mathllap{\text{and }} 
    Q_2 &= \cltd(\tau(\iota^{-2})\zeta_{00}\zeta_{10} + e^{-2}\kappa\cxltd).
\end{align*}
Then, for $k,\ell\in \Z$, we have
\begin{align*}
    e(O(2k,2\ell)) &= kQ_1 + \ell Q_2 \\
    e(O(2k+1,2\ell)) &= \clod + \zeta_{10}\zeta_{11}(kQ_1 + \ell Q_2) \\
    e(O(2k,2\ell+1)) &= \cltd + \zeta_{01}\zeta_{11}(kQ_1 + \ell Q_2) \\
    e(O(2k+1,2\ell+1)) &= \cltensd + \zeta_{01}\zeta_{10}(kQ_1 + \ell Q_2) \\
    e(\chi O(2k,2\ell)) &= e^2 + \xi(kQ_1 + \ell Q_2) \\
    e(\chi O(2k+1,2\ell)) &= \cxlod + \zeta_{00}\zeta_{01}(kQ_1 + \ell Q_2) \\
    e(\chi O(2k,2\ell+1)) &= \cxltd + \zeta_{00}\zeta_{10}(kQ_1 + \ell Q_2) \\
    e(\chi O(2k+1,2\ell+1)) &= \cxltensd + \zeta_{00}\zeta_{11}(kQ_1 + \ell Q_2).
\end{align*}
\end{proposition}

\begin{proof}
Let $\pi_1,\pi_2\colon BT^2\to BT^1$ be the two projections, so $\omega_1 = \pi_1^*\omega$ and
$\omega_2 = \pi_2^*\omega$. Then
\begin{align*}
    Q_1 &= \pi_1^* Q = \pi_1^* e(O(2)) = e(O(2,0)) \\
\intertext{and}
    Q_2 &= \pi_2^* Q = \pi_2^* e(O(2)) = e(O(0,2)).
\end{align*}

We now apply Proposition~\ref{prop:tensor w fixed} to the bundles $O(m,n)$ and $O(2k,2\ell)$,
to get
\begin{equation}\label{eqn:Euler induction}
    e(O(m+2k,n+2\ell)) = e(O(m,n)) + \zeta_1 e(O(2k,2\ell))
\end{equation}
where $\zeta_1$ is the pullback of the element of the same name in the cohomology of $BU(1)$ along the classifying
map for $O(m,n)$, hence
\[
    \zeta_1
        = \begin{cases}
            1 & \text{if $m$ and $n$ are even} \\
            \zeta_{10}\zeta_{11} & \text{if $m$ is odd and $n$ is even} \\
            \zeta_{01}\zeta_{11} & \text{if $m$ is even and $n$ is odd} \\
            \zeta_{01}\zeta_{10} & \text{if $m$ and $n$ are odd.}
        \end{cases}
\]
Considering the case when $m$ and $n$ are even and using induction on $k$ and $\ell$, (\ref{eqn:Euler induction}) shows that
\[
    e(O(2k,2\ell)) = ke(O(2,0)) + \ell e(O(0,2)) = kQ_1 + \ell Q_2.
\]
We now specialize (\ref{eqn:Euler induction}) to the case where each of $m$ and $n$ is either 0 or 1, where we note that
\begin{align*}
    e(O(0,0)) &= 0 \\
    e(O(1,0)) &= \clod \\
    e(O(0,1)) &= \cltd \\
    e(O(1,1)) &= \cltensd
\end{align*}
The first four calculations of the proposition now follow.

For the remaining four, we apply Proposition~\ref{prop:tensor w fixed} to the pair
$\chi O(m,n)$ and $O(2k,2\ell)$ to get
\[
    e(\chi O(m+2k,n+2\ell)) = e(\chi O(m,n)) + \zeta_1 e(O(2k,2\ell))
\]
where now
\[
    \zeta_1
        = \begin{cases}
            \xi & \text{if $m$ and $n$ are even} \\
            \zeta_{00}\zeta_{01} & \text{if $m$ is odd and $n$ is even} \\
            \zeta_{00}\zeta_{10} & \text{if $m$ is even and $n$ is odd} \\
            \zeta_{00}\zeta_{11} & \text{if $m$ and $n$ are odd.}
        \end{cases}
\]
The rest of the argument is similar to the case above.
\end{proof}

\begin{remarks}\label{rem:missing relation}\ 
\begin{enumerate}
\item A variation of this proof could have been used to derive all of the Euler classes
calculated in \cite[Proposition~6.5]{CHTFiniteProjSpace} from the calculation of $e(O(2))$.

\item Equation (\ref{eqn:Q}) needs to be included among the relations used in the approach
discussed in Remark~\ref{rem:global approach}, as the maps used there are classifying maps
similar to those used in the proof above,
and (\ref{eqn:Q}) is needed to determine all of their images.
\end{enumerate}
\end{remarks}

\section{$B T^2$ as a projective bundle over $B U(2)$}\label{sec:projective}

We have the map
\[
    s\colon B T^2 \to B U(2)
\]
classifying the sum $\omega_1\dirsum\omega_2$.
We studied the cohomology of $B U(2)$ in \cite{CH:bu2}; as there, we will write
$\omega = \omega(2)$ for the tautological two-plane bundle over $B U(2)$.
As is familiar nonequivariantly, $s$ is equivalent to the projective bundle $\PP(\omega)\to BU(2)$
with $\omega_1$, say, being the tautological line bundle over $\PP(\omega)$ and $\omega_2$ its quotient or complementary bundle
$s^*(\omega)/\omega_1$.

Nonequivariantly, we have the projective bundle formula that tells us that the cohomology
of $BT^2$ is free as a module over the cohomology of $BU(2)$.
Moreover, we know that the cohomology of $BU(2)$ is exactly the symmetric part of the cohomology
of $BT^2 = BT^1\times BT^1$ under the action that exchanges the two factors.
As we shall see, neither of these facts carries over completely to the equivariant context.

Note, however, that we do have that
\[
    RO(\Pi BU(2)) \iso RO(\Pi BT^2)^{\Sigma_2}
\]
via $s^*$, as $s^*$ is injective and the image of $s^*$ and the symmetric part
of $RO(\Pi BT^2)$ both consist of the gradings
$a+b\sigma + \sum_{ij} m_{ij}\Omega_{ij}$ with $m_{01} = m_{10}$.

To specify the map 
\[
 s^*\colon  H_\GG^{RO(\Pi BU(2))}(B U(2)_+) \to  H_\GG^{RO(\Pi BU(2))}(B T^2_+),
\]
it suffices to compute the image of each generator.

\begin{proposition}
We have
\begin{align*}
 s^*\zeta_0 &= \zeta_{00} \\
 s^*\zeta_1 &= \zeta_{01}\zeta_{10} \\
 s^*\zeta_2 &= \zeta_{11} \\
 s^*\cl &= \cltens \\
 s^*\cxl &= \cxltens \\
 s^*\cw &= \clo\clt \\
 s^*\cxw &= \cxlo\cxlt.
\end{align*}
\end{proposition}

\begin{proof}
These calculations follow from general principles: 
The $\zeta$s reflect the mapping of fixed-set components.
The other elements are all Euler classes, so we have, for example,
\begin{align*}
 s^*\cw &= s^*e(\omega) = e(s^*\omega) \\
 	&= e(\omega_1\dirsum\omega_2) = e(\omega_1)e(\omega_2) = \clo \clt  \\
\intertext{and}
 s^*\cl  &= s^*e(\Lambda^2\omega) = e(s^*\Lambda^2\omega) = e(\Lambda^2s^*\omega) \\
    &= e(\Lambda^2(\omega_1\dirsum\omega_2)) = e(\omega_1\tensor\omega_2) = \cltens . \qedhere
\end{align*}

\end{proof}

To see that the cohomology of $B U(2)$ is not simply the symmetric part of the cohomology
of $B T^2$, we have the following calculation, which shows that the cohomology
of $B U(2)$ does not inject into the cohomology of $B T^2$.

\begin{example}\label{ex:not injective}
The nonzero element 
\[
 \tau(\iota^{-3}) \zeta_0^2\zeta_1^2 \cw \in  H_\GG^{3 + \sigma + \Omega_1}(B U(2)_+)
\]
maps to 0 in $ H_\GG^{3+\sigma+\Omega_{01}+\Omega_{10}}(B T^2_+)$:
We first compute
\begin{align*}
 s^*(\zeta_0^2\zeta_1^2 \cw)
  &= \zeta_{00}^2\zeta_{01}^2\zeta_{10}^2 \clo \clt  \\
  &= \xi\zeta_{01}^2\clo\cxlt + e^2\zeta_{00}\zeta_{01}^2\zeta_{10}\clo.
\end{align*}
We now observe that both
\begin{align*}
 \tau(\iota^{-3}) \cdot \xi &= 0 \\
\intertext{and}
 \tau(\iota^{-3}) \cdot e^2 &= 0,
\end{align*}
so $\tau(\iota^{-3}) \zeta_0^2\zeta_1^2 \cw$ maps to 0 as claimed.
However, $\zeta_0^2\zeta_1^2 \cw$ is an element of the basis
given in \cite{CH:bu2} and $\tau(\iota^{-3})\neq 0$, so 
$\tau(\iota^{-3}) \zeta_0^2\zeta_1^2 \cw \neq 0$.

This example also shows that $\eta$ need not be injective in odd gradings
(see Definition~\ref{def:even grading}).
(Hence, the diagrams used in the proof of Proposition~\ref{prop:symmetric} below
need not be pullbacks in such gradings.)
Using the calculations of $\eta$ from \cite{CH:bu2}, we can see that
$\eta(\tau(\iota^{-3})\zeta_0^2\zeta_1^2\cw) = 0$.
Simpler examples of this phenonemon exist, but this is the simplest we have found
also showing that $s^*$ is not injective in general.

\end{example}

The following result shows that the problem occurs only in odd gradings.

\begin{proposition}\label{prop:symmetric}
Restricted to the subgroup of even gradings $\ROev(\Pi BU(2))$, $s^*$ gives an isomorphism
\[
    H_\GG^{\ROev(\Pi BU(2))}(B_G U(2)_+) \iso H_\GG^{\ROev(\Pi BU(2))}(B_G T^2_+)^{\Sigma_2}.
\]
\end{proposition}

\begin{proof}
Consider the following diagram of groups graded on $\ROev(\Pi BU(2))$ 
(which has been replaced by ``$*_e$'' for clarity of the diagram).
By \cite[Corollaries~2.3 and~2.4]{CH:bu2}, the front and back faces are both pullback squares
in which the horizontal arrows are injections.
\[
    \xymatrix@C=-4em{
    & H_\GG^{*_e}(B T^2_+) \ar[rr] \ar[dd]|\hole && H_\GG^{*_e}((B T^2)^\GG_+) \ar[dd] \\
    H_\GG^{*_e}(B U(2)_+) \ar[ur] \ar[rr] \ar[dd] && H_\GG^{*_e}(B U(2)^\GG_+) \ar[ur] \ar[dd] \\
    & H_\GG^{*_e}(B T^2_+\smsh E\GG_+) \ar[rr]|(.465)\hole && H_\GG^{*_e}((B T^2)^\GG_+\smsh E\GG_+) \\
    H_\GG^{*_e}(B U(2)_+\smsh E\GG_+) \ar[ur] \ar[rr] && H_\GG^{*_e}(B U(2)^\GG_+\smsh E\GG_+) \ar[ur]
    }
\]
Taking the $\Sigma_2$-fixed part of such pullback diagrams results in pullback diagrams, so it suffices to
show that $s$ induce isomorphisms
\begin{align}
    H_\GG^{\ROev(\Pi BU(2))}(B U(2)_+\smsh E\GG_+) &\iso H_\GG^{\ROev(\Pi BU(2))}(B T^2_+\smsh E\GG_+)^{\Sigma_2}
        \label{eqn:lower left}\\
    H_\GG^{\ROev(\Pi BU(2))}(B U(2)^\GG_+) &\iso H_\GG^{\ROev(\Pi BU(2))}((B T^2)^\GG_+)^{\Sigma_2}
        \label{eqn:upper right}\\
\intertext{and}
    H_\GG^{\ROev(\Pi BU(2))}(B U(2)^\GG_+\smsh E\GG_+) &\iso H_\GG^{\ROev(\Pi BU(2))}((B T^2)^\GG_+\smsh E\GG_+)^{\Sigma_2}.
        \label{eqn:lower right}
\end{align}
For (\ref{eqn:lower left}), the inclusions $B^0_+\to B U(2)_+$ and $B^{00}_+\to B T^2_+$ are nonequivariant equivalences,
hence induce equivariant equivalences on smashing with $E\GG_+$.
Since
\begin{align*}
    H_\GG^{RO(\Pi BU(2))}(B^0_+\smsh E\GG_+) &\iso H^{RO(\Pi BU(2))}(BU(2)_+;\Z)\tensor H_\GG^{RO(\GG)}(E\GG_+) \\
\intertext{and}
    H_\GG^{RO(\Pi BU(2))}(B^{00}_+\smsh E\GG_+) &\iso H^{RO(\Pi BU(2))}(BT^2_+;\Z)\tensor H_\GG^{RO(\GG)}(E\GG_+)
\end{align*}
by \cite[Proposition~7.2]{Co:InfinitePublished},
isomorphism (\ref{eqn:lower left}) follows from the nonequivariant result that
\[
    H^\Z(BU(2)_+) \iso H^\Z(BT^2_+)^{\Sigma_2}
\]
is the ring of symmetric polynomials in two variables.

For (\ref{eqn:upper right}), we can look at what happens to each component of $B U(2)^\GG$.
We consider the three induced maps
\begin{align*}
    s^0\colon BT^2 \hmtpc B^{00} &\to B^0 \hmtpc BU(2) \\
    s^1\colon BT^2\disjunion BT^2 \hmtpc B^{01}\disjunion B^{10} &\to B^1 \hmtpc BT^2 \\
    s^2\colon BT^2 \hmtpc B^{11} &\to B^2 \hmtpc BU(2).
\end{align*}
$s^0$ and $s^2$ are the usual maps and the nonequivariant result applies to show that
$(s^0)^*$ and $(s^2)^*$ are each the inclusion of the $\Sigma_2$-fixed part.
$s^1$ can be taken to be the identity on the first copy of $BT^2$ and the swap map on the second copy;
$\Sigma_2$ acts on the source by interchanging the two copies of $BT^2$ while also swapping the factors.
We can then see that $(s^1)^*$ is inclusion of the $\Sigma_2$-fixed part.

Finally, (\ref{eqn:lower right}) follows on taking (\ref{eqn:upper right}) and tensoring
with $H_\GG^{RO(\GG)}(E\GG_+)$.
\end{proof}

\begin{remark}\label{rem:general splitting principle}
The argument generalizes to show that the cohomology of $B U(n)$ is the symmetric
part of the cohomology of $B T^n$, when restricted to even gradings.
However, we currently do not have a good understanding of either cohomology when $n\geq 3$,
outside of the gradings covered by Proposition~\ref{prop:kunneth}.
\end{remark}

\begin{remark}
This result tells us that $H_\GG^{\ROev(\Pi BU(2))}(B_G T^2_+)^{\Sigma_2}$ is free as a module
over $H_\GG^{RO(\GG)_e}(S^0)$, the cohomology of a point in even gradings.
However, Example~\ref{ex:not injective} shows that this is false if we do not restrict to even gradings.
The element $s^*(\zeta_0^2\zeta_1^2\cw)$ generates a
``fake'' copy of $\HS$: It looks like $\HS$ in even gradings, but not in all gradings.
\end{remark}


\begin{remark}
Proposition~\ref{prop:symmetric} allows us to redo the calculation of the dual class $\cwd$
from \cite{CH:bu2}. We have
\begin{align*}
    s^*(\cwd) &= \clod\cltd \\
    &= (1-\epsilon_1)(1-\epsilon_2)\clo\clt &&\text{by Proposition~\ref{prop:duals}} \\
    &= (1-e^{-2}\kappa\zeta_{00}\zeta_{11}\cltens)\clo\clt &&\text{as in the proof of \ref{prop:duals}} \\
    &= s^*((1-\epsilon_\lambda)\cw)
\end{align*}
using the notation of \cite{CH:bu2}. Because we are in an even grading, $s^*$ is injective, hence
\[
    \cwd = (1-\epsilon_\lambda)\cw,
\]
matching the result proven in \cite[Proposition~8.5]{CH:bu2} by a different method.
\end{remark}

Nonequivariantly, we know that the cohomology of $BT^2$ is a free module over the cohomology of
$BU(2)$, of rank 2. 
Example~\ref{ex:not injective} shows that this cannot be true equivariantly if we allow all gradings,
but leaves open the possibility that it might be true if we restrict to even gradings.
However, we shall show that the equivariant cohomology of $B T^2$ is not free over the
cohomology of $B U(2)$ even with that restriced grading.

Before showing that failure, we first look at what we can say positively.
Part of the structure is given by the following result, which we also use
in \S\ref{sec:pushforward}.
We continue to restrict the grading on all modules in sight to $RO(\Pi BU(2))$,
as that will suffice for the applications we have in mind.

\begin{proposition}\label{prop:additive over bu2}
As an $H_\GG^{RO(\Pi BU(2))}(B U(2)_+)$-module,
$H_\GG^{RO(\Pi BU(2))}(B T^2_+)$ is generated by
$1$, $\zeta_{01}\clo$, $\zeta_{10}\cxlo$, and $\clo\cxlo$.
As an $H_\GG^{RO(\Pi BU(2))}(B U(2)_+)$-algebra, it is generated by
$\zeta_{01}\clo$, $\zeta_{10}\cxlo$, and $\clo\cxlo$.
\end{proposition}

\begin{proof}
We know that, as a ring, $H_\GG^{RO(\Pi T)}(B T^2_+)$ is generated by
$\zeta_{ij}$, $\clo$, $\cxlo$, $\clt$, $\cxlt$, $\cltens$, and $\cxltens$.
We consider what monomials in these elements we need in order to generate the cohomology in
gradings in $RO(\Pi B U(2))$.
The elements $\zeta_{00}$, $\zeta_{11}$, $\cltens$, and $\cxltens$ are in the image of $s^*$,
so we do not need to include them when looking for module generators. That leaves us to consider monomials
of the form
\[
    \zeta_{01}^{m_1}\zeta_{10}^{m_2}\clo^{k_1}\cxlo^{\ell_1}\clt^{k_2}\cxlt^{\ell_2}.
\]
Because $\zeta_{01}\zeta_{10} = s^*(\zeta_1)$, $\clo\clt = s^*(\cw)$, and $\cxlo\cxlt = s^*(\cxw)$,
we need only consider cases where $m_1m_2 = 0$, $k_1k_2 = 0$, and $\ell_1\ell_2 = 0$.
Further, in order for the monomial above to have grading in $RO(\Pi B U(2))$, we must have
\[
    m_1 - k_1 + \ell_1 = m_2 - k_2 + \ell_2.
\]
If we now work through the eight possible cases (of which of $m_1$ and $m_2$ is nonzero, etc.), 
we see that each monomial having grading in $RO(\Pi BU(2))$ is a product of powers of
the elements $\zeta_{01}\clo$, $\zeta_{10}\cxlo$, $\zeta_{10}\clt$, $\zeta_{01}\cxlt$,
$\clo\cxlo$, and $\clt\cxlt$. 
We can remove $\zeta_{10}\clt$, $\zeta_{01}\cxlt$, and $\clt\cxlt$ from this list because of the identities
\begin{align*}
    \zeta_{10}\clt &= s^*(\zeta_2\cl + e^{-2}\kappa\zeta_0\zeta_1\cw) - \zeta_{01}\clo \\
    \zeta_{01}\cxlt &= s^*(\zeta_0\cl + e^{-2}\kappa\zeta_1\zeta_2\cxw) - \zeta_{10}\cxlo  \quad\text{and}\\
    \clt\cxlt &= s^*(\cl\cxl - \tau(\iota^{-2})\zeta_0^2\zeta_1\cw) - \clo\cxlo.
\end{align*}
This suffices to show that $\zeta_{01}\clo$, $\zeta_{10}\cxlo$, and $\clo\cxlo$ generate multiplicatively.

Further, we have the following identities writing products of these elements as linear combinations
of them and 1:
\begin{align*}
    (\zeta_{01}\clo)^2 &= s^*(\zeta_2\cl + e^{-2}\kappa\zeta_0\zeta_1\cw)\zeta_{01}\clo - s^*(\zeta_1\cw)\cdot 1 \\
    (\zeta_{10}\cxlo)^2 &= s^*(\zeta_0\cl + e^{-2}\kappa\zeta_1\zeta_2\cxw)\zeta_{10}\cxlo - s^*(\zeta_1\cxw)\cdot 1 \\
    (\clo\cxlo)^2 &= s^*(\cl\cxl - \tau(\iota^{-2})\zeta_0^2\zeta_1\cw)\clo\cxlo - s^*(\cw\cxw)\cdot 1 \\
    \zeta_{01}\clo\cdot\zeta_{10}\cxlo &= s^*(\zeta_1)\clo\cxlo \\
    \zeta_{01}\clo\cdot\clo\cxlo &= 
        s^*(\cl\cxl - \tau(\iota^{-2})\zeta_0^2\zeta_1\cw)\zeta_{01}\clo + s^*(\cw)\zeta_{10}\cxlo \\
        &\qquad -s^*(\zeta_0\cl\cw + e^{-2}\kappa\zeta_1\zeta_2\cw\cxw)\cdot 1 \\
    \zeta_{10}\cxlo\cdot\clo\cxlo &= 
        s^*(\cl\cxl - \tau(\iota^{-2})\zeta_0^2\zeta_1\cw)\zeta_{10}\cxlo + s^*(\cxw)\cdot\zeta_{01}\clo \\
        &\qquad -s^*(\zeta_1\cl\cxw + e^{-2}\kappa\zeta_0\zeta_1\cw\cxw)\cdot 1 
\end{align*}
Thus, the elements 1, $\zeta_{01}\clo$, $\zeta_{10}\cxlo$, and $\clo\cxlo$ generate additively.
\end{proof}

\begin{remark}
The relation 
\[
    (\zeta_{01}\clo)^2 - s^*(\zeta_2\cl + e^{-2}\kappa\zeta_0\zeta_1\cw)\zeta_{01}\clo + s^*(\zeta_1\cw)\cdot 1 = 0
\]
used in the proof reduces to the nonequivariant relation
\[
    x_1^2 - c_1 x_1 + c_2 = 0
\]
given by the projective bundle formula. 
\end{remark}

\begin{remark}\label{rem:relations}
Although the set $\{ 1, \zeta_{01}\clo, \zeta_{10}\cxlo, \clo\cxlo \}$ generates additively, it does
not form a basis. There are a number of relations, including some not-at-all-obvious ones like
\[
    s^*(\zeta_0\zeta_1\cxl)\zeta_{01}\clo = \xi\clo\cxlo + s^*(\zeta_0^2\zeta_1\cw)\cdot 1.
\]
We will not attempt to write down a complete set of relations.
\end{remark}

\subsection{It is not free}
Remark~\ref{rem:relations} tells us that an easily found set of generators
does not form a basis over $H_\GG^{RO(\Pi BU(2))}(B U(2)_+)$. 
In fact, there can be no basis.
To see that, we first simplify the structure greatly by doing the following.

\begin{definition}
Let $N = N^{RO(\GG)}\subset \HS$ be the ideal generated by the elements in non-zero gradings.
\end{definition}

\begin{proposition}\label{prop:ideal structure}
$Q = \HS/N$ is $\Z/2$ concentrated in grading $0$.
\end{proposition}

\begin{proof}
All of $\HS/N$ lives in grading $0$ by the definition of $N$.
In grading 0, where $H_\GG^0(S^0) = A(\GG)$, we have the elements $g = \xi\cdot\tau(\iota^{-2}) \in N$
and $\kappa = e^2\cdot e^{-2}\kappa\in N$, hence also $2 = \kappa + g\in N$.
On the other hand, we cannot have $1\in N$, because the only units in $\HS$ are $\pm 1$ and $\pm(1-\kappa)$,
which live in grading 0. Therefore, $\HS/N$ must have at least one nonzero
element, the image of 1, but it has no others because $g\in N$ and $2\in N$
and $A(\GG)/\rels{2,g} \iso \Z/2$.
\end{proof}

We can now mod out the action of $N$ on all of the modules in sight, meaning taking
$-\tensorS \HS/N$, to get modules over $Q$. 
Free modules over $\HS$ will become free modules over $Q$ with the same bases.
Further, if $H_\GG^{RO(\Pi BU(2))}(B T^2_+)$ were free over $H_\GG^{RO(\Pi BU(2))}(B U(2)_+)$,
then $H_\GG^{RO(\Pi BU(2))}(B T^2_+)/N$ would remain free over
$H_\GG^{RO(\Pi BU(2))}(B U(2)_+)/N$.
Note that, since $N$ contains all the elements in the cohomology of a point with odd gradings,
the same statement is true if we restrict the grading to $\ROev(\Pi BU(2))$.

\begin{proposition}
$H_\GG^{RO(\Pi BU(2))}(B T^2_+)/N$ is not free
$H_\GG^{RO(\Pi BU(2))}(B U(2)_+)/N$.
\end{proposition}

\begin{proof}
As an algebra over $Q$, $H_\GG^\gr(B T^2_+)/N$ is described
as in Theorem~\ref{thm:main}, except that the relations simplify to the following:
\begin{align*}
    {\textstyle \prod_{i,j}\zeta_{ij}} &= 0 \\
    \zeta_{10}\zeta_{11}\cxlo &= \zeta_{00}\zeta_{01}\clo \\
    \zeta_{01}\zeta_{11}\cxlt &= \zeta_{00}\zeta_{10}\clt \\
    \zeta_{00}\cltens &= \zeta_{10}\cxlo + \zeta_{01}\cxlt \\
    \zeta_{11}\cltens &= \zeta_{01}\clo + \zeta_{10}\clt \\
    \zeta_{01}\cxltens &= \zeta_{11}\cxlo + \zeta_{00}\clt \\
    \zeta_{10}\cxltens &= \zeta_{00}\clo + \zeta_{11}\cxlt \\
    \cltens\cxltens &= \clo\cxlo + \clt\cxlt.
\end{align*}
Similarly, $H_\GG^{RO(\Pi BU(2))}(B U(2)_+)/N$ is generated by
\begin{align*}
    \zeta_0 &= \zeta_{00} \\
    \zeta_1 &= \zeta_{01}\zeta_{10} \\
    \zeta_2 &= \zeta_{11} \\
    \cl &= \cltens \\
    \cxl &= \cxltens \\
    \cw &= \clo\clt \\
    \mathllap{\text{and }}\cxw &= \cxlo\cxlt
\end{align*}
subject to the relations
\begin{align*}
    \zeta_0\zeta_1\zeta_2 &= 0 \\
    \zeta_1\cxl &= \zeta_0\zeta_2\cl \\
    \zeta_2^2\cxw &= \zeta_0^2\cw.
\end{align*}

Now suppose $H_\GG^{RO(\Pi BU(2))}(B T^2_+)/N$ were free over
$H_\GG^{RO(\Pi BU(2))}(B U(2)_+)/N$. If a basis consisted of polynomials none
of which contained 1 as a term, the relations above show that no linear combination
could equal 1. Hence any basis must include an element having 1 as a term.
But, 1 is the only element in $H_\GG^{RO(\Pi BU(2))}(B T^2_+)/N$ in grading 0,
so an element having 1 as a term must in fact simply equal 1.
In other words, any basis would have to have 1 as one of its elements.

Now let's look at what happens in the $RO(\GG)$ grading. 
Here is what the cohomology of $B T^2$ looks like, with the numbers indicating
the rank of each $\Z/2$-module. (The grid lines are spaced every 2; see \S\ref{sec:bases}.)
\begin{center}
\begin{tikzpicture}[scale=0.4] 
    \draw[step=1cm,gray,very thin] (-0.9,-0.9) grid (4.9,6.9);
    \draw[thick] (-1,0) -- (5,0);
    \draw[thick] (0,-1) -- (0,7);
    \node[right] at (5,0) {$a$};
    \node[above] at (0,7) {$b\sigma$};

    \node[fill=white,scale=0.7] at (0,0) {1}; \draw (0,0) circle(0.38cm);
    \node[fill=white,scale=0.7] at (0,1) {2}; \draw (0,1) circle(0.38cm);
    \node[fill=white,scale=0.7] at (0,2) {1}; \draw (0,2) circle(0.38cm);

    \node[fill=white,scale=0.7] at (1,1) {2}; \draw (1,1) circle(0.38cm);
    \node[fill=white,scale=0.7] at (1,2) {4}; \draw (1,2) circle(0.38cm);
    \node[fill=white,scale=0.7] at (1,3) {2}; \draw (1,3) circle(0.38cm);

    \node[fill=white,scale=0.7] at (2,2) {3}; \draw (2,2) circle(0.38cm);
    \node[fill=white,scale=0.7] at (2,3) {6}; \draw (2,3) circle(0.38cm);
    \node[fill=white,scale=0.7] at (2,4) {3}; \draw (2,4) circle(0.38cm);

    \node[fill=white,scale=0.7] at (3,3) {4}; \draw (3,3) circle(0.38cm);
    \node[fill=white,scale=0.7] at (3,4) {8}; \draw (3,4) circle(0.38cm);
    \node[fill=white,scale=0.7] at (3,5) {4}; \draw (3,5) circle(0.38cm);

    \node[fill=white,scale=0.8,rotate=45] at (4,4) {\dots};
    \node[fill=white,scale=0.8,rotate=45] at (4,5) {\dots};
    \node[fill=white,scale=0.8,rotate=45] at (4,6) {\dots};

\end{tikzpicture}
\end{center}
Since 1 is a basis element,
its multiples must form a summand of the $RO(\GG)$-graded part as
module over $H_\GG^{RO(\Pi BU(2))}(B U(2)_+)/N$.
Here is how those multiples are distributed (this is how the basis elements of $H_\GG^{RO(\GG)}(BU(2)_+)$ are distributed):
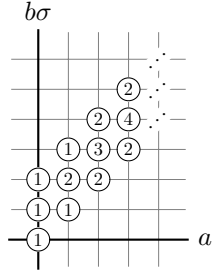
\begin{center}
\begin{tikzpicture}[scale=0.4] 
    \draw[step=1cm,gray,very thin] (-0.9,-0.9) grid (4.9,6.9);
    \draw[thick] (-1,0) -- (5,0);
    \draw[thick] (0,-1) -- (0,7);
    \node[right] at (5,0) {$a$};
    \node[above] at (0,7) {$b\sigma$};

    \node[fill=white,scale=0.7] at (0,0) {1}; \draw (0,0) circle(0.38cm);
    \node[fill=white,scale=0.7] at (0,1) {1}; \draw (0,1) circle(0.38cm);
    \node[fill=white,scale=0.7] at (0,2) {1}; \draw (0,2) circle(0.38cm);

    \node[fill=white,scale=0.7] at (1,1) {1}; \draw (1,1) circle(0.38cm);
    \node[fill=white,scale=0.7] at (1,2) {2}; \draw (1,2) circle(0.38cm);
    \node[fill=white,scale=0.7] at (1,3) {1}; \draw (1,3) circle(0.38cm);

    \node[fill=white,scale=0.7] at (2,2) {2}; \draw (2,2) circle(0.38cm);
    \node[fill=white,scale=0.7] at (2,3) {3}; \draw (2,3) circle(0.38cm);
    \node[fill=white,scale=0.7] at (2,4) {2}; \draw (2,4) circle(0.38cm);

    \node[fill=white,scale=0.7] at (3,3) {2}; \draw (3,3) circle(0.38cm);
    \node[fill=white,scale=0.7] at (3,4) {4}; \draw (3,4) circle(0.38cm);
    \node[fill=white,scale=0.7] at (3,5) {2}; \draw (3,5) circle(0.38cm);

    \node[fill=white,scale=0.8,rotate=45] at (4,4) {\dots};
    \node[fill=white,scale=0.8,rotate=45] at (4,5) {\dots};
    \node[fill=white,scale=0.8,rotate=45] at (4,6) {\dots};

\end{tikzpicture}
\end{center}
From these diagrams we can see that the other summand must contain an element
in grading $2\sigma$ other than the multiple of 1 in that grading, which is
\[
    \zeta_0\zeta_2\cl\cdot 1 = \zeta_{00}\zeta_{11}\cltens = \zeta_{00}\zeta_{01}\clo + \zeta_{00}\zeta_{10}\clt.
\]
The two possibilities for this other element are $\zeta_{00}\zeta_{01}\clo$ or $\zeta_{00}\zeta_{10}\clt$.
Since they differ by an element of $H_\GG^{RO(\Pi BU(2))}(B U(2)_+)$,
we need only consider one of them, say $\zeta_{00}\zeta_{01}\clo = \zeta_0\zeta_{01}\clo$.
As shown, this is a multiple of $\zeta_{01}\clo$, which must therefore be another basis element.
Because $\grad\zeta_{01}\clo = \Omega_1 + \Omega_2 + 2$,
its multiples in the $RO(\GG)$ grading will give a copy of $H_\GG^{\Omega_0 + RO(\GG)}(B U(2)_+)$,
and they will be arranged as follows.
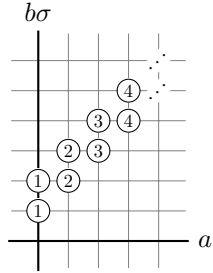
\begin{center}
\begin{tikzpicture}[scale=0.4] 
    \draw[step=1cm,gray,very thin] (-0.9,-0.9) grid (4.9,6.9);
    \draw[thick] (-1,0) -- (5,0);
    \draw[thick] (0,-1) -- (0,7);
    \node[right] at (5,0) {$a$};
    \node[above] at (0,7) {$b\sigma$};

    \node[fill=white,scale=0.7] at (0,1) {1}; \draw (0,1) circle(0.38cm);
    \node[fill=white,scale=0.7] at (0,2) {1}; \draw (0,2) circle(0.38cm);

    \node[fill=white,scale=0.7] at (1,2) {2}; \draw (1,2) circle(0.38cm);
    \node[fill=white,scale=0.7] at (1,3) {2}; \draw (1,3) circle(0.38cm);

    \node[fill=white,scale=0.7] at (2,3) {3}; \draw (2,3) circle(0.38cm);
    \node[fill=white,scale=0.7] at (2,4) {3}; \draw (2,4) circle(0.38cm);

    \node[fill=white,scale=0.7] at (3,4) {4}; \draw (3,4) circle(0.38cm);
    \node[fill=white,scale=0.7] at (3,5) {4}; \draw (3,5) circle(0.38cm);

    \node[fill=white,scale=0.8,rotate=45] at (4,5) {\dots};
    \node[fill=white,scale=0.8,rotate=45] at (4,6) {\dots};

\end{tikzpicture}
\end{center}
But this will clearly give too may elements in grading $4\sigma$.
In fact, there we will see the linear relation
\[
    \zeta_0\zeta_1\cxl\cdot \zeta_{01}\clo = \zeta_0^2\zeta_1\cw \cdot 1
\]
that follows from Proposition~\ref{prop:additive over bu2} or can be checked directly
from the relations above.
Therefore, there can be no basis, and
$H_\GG^{RO(\Pi BU(2))}(B T^2_+)/N$ is not a free module over
$H_\GG^{RO(\Pi BU(2))}(B U(2)_+)/N$.
\end{proof}

As per the discussion before this proposition, we get the following immediate corollary.

\begin{corollary}\label{cor:notfree}
$H_\GG^{\ROev(\Pi BU(2))}(B T^2_+)$ is not free over
$H_\GG^{\ROev(\Pi BU(2))}(B U(2)_+)$.
\qed
\end{corollary}

This is an example showing that, in general, we cannot expect the equivariant cohomology of a projective
bundle to be free over the cohomology of its base---even when restricted to even gradings---contrary 
to what happens nonequivariantly.

\section{The pushforward map to the cohomology of $B U(2)$}\label{sec:pushforward}

The bundle map $s\colon B T^2\to B U(2)$ induces not only the map $s^*$ but a transfer
or pushforward map
\[
    s_!\colon H_\GG^{RO(\Pi BU(2))}(B T^2_+) \to H_\GG^{RO(\Pi BU(2))}(B U(2)_+)
\]
that reduces grading by $\lambda$. (The fibers of $s$ are projective lines whose equivariant dimensions
are given by $\lambda$.)
With the idea that this pushforward will be useful in studying the cohomology of finite Grassmannians,
we compute it here.
It will be more convenient to work with the dual classes introduced in \S\ref{sec:units}.

Because $s_!$ is a map of $H_\GG^{RO(\Pi BU(2))}(B U(2)_+)$-modules, it suffices to find
its values on the generators 
$1$, $\zeta_{01}\clod$, $\zeta_{10}\cxlod$, and $\clod\cxlod$ implied by
Proposition~\ref{prop:additive over bu2}.
(Recall that the dual generators differ from the usual ones by units.)
We will use the fact that the map $\eta$ is injective in even gradings to reduce the calculations
to nonequivariant ones. There, we will use the following result.

\begin{proposition}\label{prop:nonequivariant pushforward}
The nonequivariant pushforward map
\[
    s_!\colon H^\Z(BT^2_+) \to H^\Z(BU(2)_+),
\]
which reduces grading by $2$, has
\[
    s_!(1) = 0 \text{ and } s_!(\xd_1) = 1.
\]
\qed
\end{proposition}

Of course, then we have
\[
    s_!s^*(a) = 0 \text{ and } s_!(s^*(a)\xd_1) = a
\]
for $a\in H^\Z(BU(2)_+)$.

Returning to the equivariant case, we let
\[
    T^i = s^{-1}(B^i) \quad i = 0, 1, 2
\]
so that
\[
    T^0 = T^{00} = BT^2 \text{ and } T^2 = T^{11} = BT^2,
\]
with trivial $\GG$-action,
while $T^1$ has nontrivial action.
Write $s^i$ for the restriction of $s$ to $T^i$. 
(Note that this $s^1$ is not the map we called $s^1$ in the proof of Proposition~\ref{prop:symmetric}.)
Then the naturality of the transfer
tells us that we have the following commutative diagram
(where we write $\star$ for $RO(\Pi BU(2))$ for brevity).
\begin{equation}\label{diag:T2 and U2}
\xymatrix{
    H_\GG^{\star}(B T^2_+) \ar[d]_{s_!} \ar[r]^-{\bar\eta}
        & H_\GG^{\star}(T^0_+)\dirsum H_\GG^\star(T^1_+)\dirsum H_\GG^\star(T^2_+) 
                \ar[d]^{s^0_!\dirsum s^1_! \dirsum s^2_!} \\
    H_\GG^{\star}(B U(2)_+) \ar@{>->}[r]_-\eta
        & H_\GG^{\star}(B^0_+)\dirsum H_\GG^\star(B^1_+)\dirsum H_\GG^\star(B^2_+)
}
\end{equation}

The restrictions
\[
    s^0\colon T^0\to B^0 \text{ and } s^2\colon T^2\to B^2
\]
are copies of the nonequivariant map $s\colon BT^2\to BU(2)$, so their pushforwards are given
essentially by Proposition~\ref{prop:nonequivariant pushforward}.

For the middle component, recall that $B^1 = \Xp\infty \times \Xq\infty = BT^2$ and write
\[
    \omega|B^1 = \mu_1\dirsum \mu_2
\]
where $\mu_1$ and $\mu_2$ are the tautological bundles over $B^1$, with $\mu_1$ having trivial
$\GG$-action but $\mu_2$ having the nontrivial $\GG$-action on each fiber.
We then have
\[
    T^1 = \PP(\mu_1\dirsum\mu_2).
\]
Although Corollary~\ref{cor:notfree} shows us that we don't have a projective bundle formula
in general, we will show that $T^1\to B^1$ is a case where we do have something like it.

The bundle $\omega_1$ over $B T^2$ restricts to the tautological line bundle over $T^1 = \PP(\mu_1\dirsum\mu_2)$,
which we call $\omega_1$ again, and $\omega_2$ is its complementary bundle.
Write $\zeta_{ij}$, $\clod$, and so on for the restrictions to the cohomology of $T^1$
of the elements with those names in the cohomology of $B T^2$.
Note that $\zeta_{00}$ and $\zeta_{11}$ are invertible because $B^1$ is disjoint from $B^0$
and $B^2$.

Because $\zeta_{00}$ and $\zeta_{11}$ are invertible, to compute
$H_\GG^{RO(\Pi BU(2))}(T^1_+)$ it suffices to compute it in $RO(\GG)$ grading. There, we have
the following.

\begin{proposition}\label{prop:t1 structure}
$H_\GG^{RO(\GG)}(T^1_+)$ is a free module over $H_\GG^{RO(\GG)}(B^1_+)$
on the two elements $1$ and $\zeta_{00}\zeta_{01}\clod$.
\end{proposition}

\begin{proof}
Write $y_1 = e(\mu_1)$ in $H_\GG^{RO(\Pi BU(2))}(B^1_+)$.
We have the following cofibration sequence over $B^1$:
\[
    \PP(\mu_1)_+ \to T^1_+ \to \susp^{\omega_1\dual\tensor\mu_1}\PP(\mu_2)_+
\]
where $\PP(\mu_1)\homeo \PP(\mu_2) \homeo B^1$.
Here, we are using that $\PP(\mu_2)$ is the zero locus of the section of 
$\omega_1\dual\tensor\mu_1 \iso \Hom(\omega_1,\mu_1)$
given by the composite
\[
    \omega_1 \to \mu_1\dirsum\mu_2 \to \mu_1,
\]
where the first map is the inclusion of the subbundle and the second is projection.
Since $\mu_1$ has trivial action on its fibers, $\grad \omega_1\dual\tensor\mu_1 = \grad\omega_1$.

Since $\PP(\mu_1)\to T^1_+$ is a section of $s^1\colon T^1\to B^1$, we get a splitting
in the $RO(\GG)$ grading, hence a split short exact sequence
\begin{equation}\label{eqn:T1 SES}
    0 \to H_\GG^{\alpha-\omega_1}(\PP(\mu_2)_+) 
        \xrightarrow{\cdot e(\omega_1\dual\tensor\mu_1)}
    H_\GG^{\alpha}(T^1_+) \to H_\GG^{\alpha}(\PP(\mu_1)_+) \to 0
\end{equation}
for $\alpha\in RO(\GG)$.
Here, we are grading the cohomology of $\PP(\mu_2)$ on $RO(\Pi BT^2)$ via the inclusion
\[
    i\colon \PP(\mu_2) = T^{10} \to T^1.
\]
With this grading, $H_\GG^{-\omega_1+RO(\GG)}(\PP(\mu_2)_+)$ is a free module over
$H_\GG^{RO(\GG)}(B^1_+)$ on the invertible element $\zeta_{00}\zeta_{01}$.

By Proposition~\ref{prop:tensor w fixed},
\[
    e(\omega_1\dual\tensor\mu_1) = \clo + \zeta_{10}\zeta_{11}(s^1)^*(y_1).
\]
It follows from (\ref{eqn:T1 SES}) that $H_\GG^{RO(\GG)}(T^1_+)$ is a free
module over $H_\GG^{RO(\GG)}(B^1_+)$ on a basis consisting of 1 and
$\zeta_{00}\zeta_{01}\clo + \xi(s^1)^*(y_1)\cdot 1$.
But we can simplify that to the basis $\{1, \zeta_{00}\zeta_{01}\clo\}$
as claimed.
\end{proof}

\begin{corollary}\label{cor:t1 pushforward}
The pushforward map
\[
    s^1_!\colon H_\GG^{RO(\Pi BU(2))}(T^1_+)\to H_\GG^{RO(\Pi BU(2))}(B^1_+)
\]
is determined by the values
\[
    s^1_!(1) = e^{-2}\kappa\zeta_0\zeta_2 \qquad\text{and}\qquad s^1_!(\zeta_{00}\zeta_{01}\clod) = \zeta_0\zeta_2.
\]
\end{corollary}

\begin{proof}
That $s^1_!$ is determined by these two values follows from the preceding proposition, so
it remains to compute $s^1_!(1)$ and $s^1_!(\zeta_{00}\zeta_{01}\clod)$.

We already noted that $s_!$ lowers grading by $\lambda = 2 + \Omega_1$, which is $2\sigma$ when restricted to $B^1$.
So we first compute
\[
    s^1_!\colon H_\GG^\alpha(T^1_+) \to H_\GG^{\alpha-2\sigma}(B^1_+)
\]
for $\alpha\in RO(\GG)$, then we will restore the full $RO(\Pi BU(2))$ grading.

Let $y_1$ and $y_2$ be the Euler classes of the two (nonequivariant) tautological bundles
over $B^1 = BT^2$.
Because $B^1$ has trivial $\GG$-action, we can see that
\[
    H_\GG^{RO(\GG)}(B^1_+) \iso \HS[y_1,y_2]
\]
where $\grad y_1 = 2 = \grad y_2$.

Because $s^1_!(1) \in H_\GG^{-2\sigma}(B^1_+)$, we must have $s^1_!(1) = ae^{-2}\kappa$ for some $a\in\Z$.
To determine $a$, we look at what happens when we take fixed points, where we will see the map
\[
    (s^1)^\GG\colon (T^1)^\GG = T^{01}\disjunion T^{10}\to B^1.
\]
This is $\id\disjunion\gamma\colon BT^2\disjunion BT^2 \to BT^2$ where $\gamma$ is the map that swaps
the two factors of $BT^2 = BT^1\times BT^1$. Pushforwards commute with taking fixed points, so we get
\[
    (s^1_!(1))^\GG = (s^1)^\GG_!(1^\GG) = (s^1)^\GG_!(1,1) = 1 + 1 = 2.
\]
On the other hand, $(e^{-2}\kappa)^\GG = 2$, so we must have $a = 1$. Thus,
$s^1_!(1) = e^{-2}\kappa$ when restricted to $RO(\GG)$ grading.

Similarly, $s^1_!(\zeta_{00}\zeta_{01}\clo)\in H_\GG^0(B^1_+)$,
so $s^1_!(\zeta_{00}\zeta_{01}\clo) = \alpha$ for some $\alpha\in A(\GG)$.
Nonequivariantly, this is the pushforward of the Euler class of the tautological
bundle of $\PP(\mu_1\dirsum\mu_2)$, which is 1, as we can see by looking at
what happens on any fiber. Thus, $\rho s^1_!(\zeta_{00}\zeta_{01}\clo) = 1$.
On the other hand, if we look at fixed points, $\omega_1^\GG$ restricts to 
the tautological line bundle $\omega_1$ on $T^{01} = BT^2$, but
restricts to the zero bundle on $T^{10}$.
The pushforward of its
Euler class on $T^{10}$ is 0 for dimensional reasons, while the pushforward
from $T^{01}$ is 1. Thus, $s^1_!(\zeta_{00}\zeta_{01}\clo)^\GG = 1$.
The only element of $A(\GG)$ that restricts to 1 and has 1 as its fixed points is
the unit 1, therefore $s^1_!(\zeta_{00}\zeta_{01}\clo)=1$.

Expanding the grading to $RO(\Pi BU(2))$, $s^1_!(1)$ should live in grading $-\lambda = -2\sigma+\Omega_0+\Omega_2$,
and we correct the grading by multiplying by the invertible element $\zeta_0\zeta_2$.
We make the same correction to $s^1_!(\zeta_{00}\zeta_{01}\clo)$, getting the values
in the statement of the corollary.
\end{proof}

\begin{theorem}\label{thm:pushforward}
The pushforward map 
\[
    s_!\colon H_\GG^{RO(\Pi BU(2))}(B T^2_+) \to H_\GG^{RO(\Pi BU(2))}(B U(2)_+)
\]
is determined by its values
\begin{align*}
    s_!(1) &= e^{-2}\kappa\zeta_0\zeta_2 \\
    s_!(\zeta_{01}\clod) &= \zeta_2 \\
    s_!(\zeta_{10}\cxlod) &= \zeta_0 \\
    s_!(\clod\cxlod) &= \cxl.
\end{align*}
\end{theorem}

\begin{proof}
That these values determine $s_!$ follows from Proposition~\ref{prop:additive over bu2}.

To check that these values are correct, recall that $\eta$ is injective in even gradings 
on the cohomology of $B U(2)$ and that we have the commutative diagram (\ref{diag:T2 and U2}).
For reference, we have
\begin{align*}
    \eta(\zeta_0) &= (\xi\zeta_1^{-1}\zeta_2^{-1}, \zeta_0, \zeta_0 ) \\
    \eta(\zeta_1) &= (\zeta_1, \xi\zeta_0^{-1}\zeta_2^{-1}, \zeta_1) \\
    \eta(\zeta_2) &= (\zeta_2, \zeta_2, \xi\zeta_0^{-1}\zeta_1^{-1}) \\
    \eta(\cld) &= (\cd_1\zeta_1, (e^2 + \xi(\xd_1+\xd_2))\zeta_0^{-1}\zeta_2^{-1}, \cd_1\zeta_1) \\
    \eta(\cxld) &= ((e^2+\xi \xd_1)\zeta_1^{-1}, (\xd_1+\xd_2)\zeta_0\zeta_2, (e^2+\xi \cd_1)\zeta_1^{-1}) \\
    \eta(\cwd) &= (\cd_2\zeta_1\zeta_2^2, \xd_1(e^2+\xi \xd_2)\zeta_0^{-1}\zeta_2, (e^4 + e^2\xi \cd_1 + \xi^2 \cd_2)\zeta_0^{-2}\zeta_1^{-1}) \\
    \eta(\cxwd) &= ((e^4 + e^2\xi \cd_1 + \xi^2 \cd_2)\zeta_1^{-1}\zeta_2^{-2}, \xd_2(e^2 + \xi \xd_1)\zeta_0\zeta_2^{-1}, \cd_2\zeta_0^2\zeta_1 )
\end{align*}

Using what we've said about the pushforward maps $s^i_!$, we get the following.
\begin{align*}
    \eta(s_!(1)) &= (s^0_!\dirsum s^1_!\dirsum s^2_!)(\bar\eta(1)) \\
        &= (s^0_!\dirsum s^1_!\dirsum s^2_!)(1,1,1) \\
        &= (0,e^{-2}\kappa\zeta_0\zeta_2,0) \\
        &= \eta(e^{-2}\kappa\zeta_0\zeta_2).
\end{align*}
Therefore, $s_!(1) = e^{-2}\kappa\zeta_0\zeta_2$ as claimed.
Similarly, we have the following.
\begin{align*}
    \eta(s_!(\zeta_{01}\clod)) &= (s^0_!\dirsum s^1_!\dirsum s^2_!)(\bar\eta(\zeta_{01}\clod)) \\
        &= (s^0_!\dirsum s^1_!\dirsum s^2_!)(\xd_1\zeta_{01}\zeta_{10}\zeta_{11}, \zeta_{01}\clod, (e^2+\xi \xd_1)\zeta_{00}^{-1}) \\
        &= (\zeta_2, \zeta_2, \xi\zeta_0^{-1}\zeta_1^{-1}) \\
        &= \eta(\zeta_2)
\end{align*}
\begin{align*}
    \eta(s_!(\zeta_{10}\cxlod)) &= (s^0_!\dirsum s^1_!\dirsum s^2_!)(\bar\eta(\zeta_{10}\cxlod)) \\
        &= (s^0_!\dirsum s^1_!\dirsum s^2_!)((e^2+\xi \xd_1)\zeta_{11}^{-1}, \zeta_{10}\cxlod, \xd_1\zeta_{00}\zeta_{01}\zeta_{10}) \\
        &= (\xi\zeta_1^{-1}\zeta_2^{-1}, \zeta_0, \zeta_0) \\
        &= \eta(\zeta_0)
\end{align*}
\begin{align*}
    \eta(s_!(\clod\cxlod)) &= (s^0_!\dirsum s^1_!\dirsum s^2_!)(\bar\eta(\clod\cxlod)) \\
        &= (s^0_!\dirsum s^1_!\dirsum s^2_!)(\xd_1(e^2+\xi \xd_1), \clod\cxlod, \xd_1(e^2+\xi \xd_1)) \\
        &= ((e^2+\xi \cd_1)\zeta_1^{-1}, (\xd_1 + \xd_2)\zeta_0\zeta_2, (e^2+\xi \cd_1)\zeta_1^{-1}) \\
        &= \eta(\cxl)
\end{align*}
For the first and last components of the last calculation, we use that
\[
    \xd_1(e^2+\xi \xd_1) = \xd_1(e^2 + \xi \cd_1) - \xi \cd_2
\]
on $T^0$ and $T^2$, hence $s^0_!(\xd_1(e^2+\xi \xd_1)) = s^2_!(\xd_1(e^2+\xi \xd_1)) = (e^2+\xi \cd_1)\zeta_1^{-1}$.
For the middle component, we use that
\[
    \clod\cxlod = \zeta_{00}\zeta_{01}\clod {[(1-\kappa) \xd_1 + \xd_2]} + e^2 \xd_1 - \xi \xd_1\xd_2
\]
on $T^1$. This can be checked using the facts that
\[
    \cxltensd = s^*((\xd_1 + \xd_2)\zeta_0\zeta_2)
    \quad\text{and}\quad
    \clod\cltd = s^*(\xd_1(e^2+\xi\xd_2)\zeta_0^{-1}\zeta_2)
\]
and the relations in the cohomology of $BT^2$ to compute that
\begin{align*}
    \zeta_{00}\zeta_{01}\clod (\xd_1+\xd_2)
    &= \zeta_{01}\zeta_{11}^{-1}\clod\cxltensd \\
    &= \zeta_{11}^{-1}\clod\cdot \zeta_{01}\cxltensd \\
    &= \dots \\
    &= \clod\cxlod + \kappa \zeta_{00}\zeta_{01}\clod \xd_1 - e^2\xd_1 + \xi\xd_1\xd_2.
\end{align*}
From the formula above for $\clod\cxlod$ it follows that
\[
    s^1_!(\clod\cxlod) = ((1-\kappa)\xd_1 + \xd_2 + \kappa \xd_1)\zeta_0\zeta_2 = (\xd_1 + \xd_2)\zeta_0\zeta_2.
\]
This completes the verification of the theorem.
\end{proof}

\section{Waner classes and a sum formula}\label{sec:waner}

In \cite{Wan:ChernClasses}, Stefan Waner suggested constructing characteristic classes as follows,
generalizing a nonequivariant construction of the Chern classes:
Given an $n$-dimensional complex vector bundle $\alpha$ over a space $B$, we have its Euler class $e(\alpha) = \cw(\alpha)$
(where $\omega = \omega(n)$),
which we take as
its ``top'' Chern class. Now consider the Gysin sequence of $\alpha$, the long exact sequence induced by the
cofibration $S(\alpha)_+\to D(\alpha)_+\to T(\alpha)$ over $B$, where $T(\alpha)$ is the (parametrized) Thom space.
Part of the Gysin sequence is the following:
\[
 0 =  H_\GG^{-2}(B_+) \xrightarrow{\cdot e(\alpha)}
 	 H_\GG^{\alpha-2}(B_+) \to  H_\GG^{\alpha-2}(S(\alpha)_+)
	\to  H_\GG^{-1}(B_+) = 0.
\]
When we pull the bundle $\alpha$ back to $S(\alpha)$, the result splits as $\alpha'\dirsum\C$
for some $(n-1)$-dimensional complex bundle $\alpha'$, so we have
the class $e(\alpha') \in  H_\GG^{\alpha-2}(S(\alpha)_+)\iso  H_\GG^{\alpha-2}(B_+)$.
We write $c_{\omega-2}(\alpha)$ for this class in $ H_\GG^{\alpha-2}(B_+)$. Continuing inductively,
we get a series of classes $c_{\omega-2k}$, $0\leq k \leq n$,
which we call the \emph{Waner classes}.
These come from universal classes $c_{\omega-2k} \in H_\GG^{\omega-2k}(BU(n)_+)$.

Clearly, $c_{\omega-2k}$ restricts nonequivariantly to the Chern
class $c_{n-k}$, because the construction above produces the Chern classes when done non\-equivariantly.
This raises the question: What are the classes $c_{\omega-2}$ and $c_{\omega-4}\in  H_\GG^{RO(\Pi BU(2))}(B U(2)_+)$?
We compute them by first showing we have a general sum formula for the Waner classes.

Nonequivariantly, the Whitney sum formula allows us to, for example, calculate
the first Chern class of a sum of bundles in terms of the first Chern classes of the individual bundles.
Equivariantly, if $\omega_1$ and $\omega_2$ are the two tautological bundles over $BT^2$, we have that
\[
    c_\lambda(\omega_1\dirsum\omega_2) = \cltens
\]
in the cohomology of $B T^2$,
which is related to, but cannot be written in terms of, $\clo$ and $\clt$.
So we cannot expect a sum formula that would allow us to express $c_\lambda(\omega_1\dirsum\omega_2)$ 
in general in terms of the Euler classes of $\omega_1$ and $\omega_2$.
What we do have is the following.

\begin{definition}
If $\alpha$ is a complex $n$-plane bundle, the \emph{total Waner class} is the formal sum
\[
    W(\alpha) = \sum_{k=0}^n c_{\omega-2k}(\alpha)t^{n-k}.
\]
\end{definition}

\begin{proposition}[Whitney Sum]\label{prop:Whitney sum}
If $\alpha_1$ and $\alpha_2$ are complex vector bundles over $B$, then
\[
    W(\alpha_1\dirsum \alpha_2) = W(\alpha_1)W(\alpha_2).
\]
\end{proposition}

\begin{proof}
We use Remark~\ref{rem:general splitting principle} as a general splitting principle
and Proposition~\ref{prop:kunneth} for computations.
Consider an $n$-dimensional complex vector bundle $\alpha$.
The classes $c_{\omega-2k}(\alpha)$
are pullbacks of the classes $c_{\omega(n)-2k}$ in the cohomology of $BU(n)$,
so we may assume we are in the universal case $\alpha = \omega(n) = \omega$ over $BU(n)$.
The classes $c_{\omega-2k}$ lie in even gradings, so lie in groups that can be identified with
symmetric elements in the cohomology of $BT^n$.
Further, their gradings lie in $SRO(\Pi BT^n)$.
So let us work entirely in the cohomology of $BT^n$ and write
$\omega = \omega_1\dirsum\cdots\dirsum\omega_n$.

Now we claim that, in the cohomology of $BT^n$, we have
\begin{equation}\label{eqn:roots}
    W(\omega) = \prod_{i=1}^n(\zeta^{\omega_i-2} + c_{\omega_i}t),
\end{equation}
from which the multiplicativity of the proposition follows in general.
Using Proposition~\ref{prop:kunneth}, we can see that the elements $c_{\omega-2k}$
lie in gradings in which no basis elements lie below them,
hence, when expressed as linear combinations of basis elements, they will have coefficients
that can be detected by looking at $\rho$, the restriction to nonequivariant cohomology,
and $(-)^\GG$, taking fixed sets.
Thus, we can verify (\ref{eqn:roots}) by checking that it is true on applying $\rho$ and
on applying $(-)^\GG$.

On applying $\rho$, equation (\ref{eqn:roots}) restricts to the nonequivariant Whitney sum formula,
which we know to be true, so it remains to check that (\ref{eqn:roots})
is true on taking fixed points.
Write
\[
    (BT^n)^\GG = \Disjunion_{I\in (\Z/2)^n}(BT^n)^I
\]
where $(BT^n)^I \homeo BT^n$ is the component whose $i$th factor is $\Xp\infty$ if $I_i=0$ and $\Xq\infty$ if $I_i=1$.
On taking fixed sets, the $I$th component of (\ref{eqn:roots}) is
\[
    c(\omega^\GG|(BT^n)^I) = \prod_{I_i = 0}(1 + x_i t),
\]
which is true
because $\omega_i^\GG|(BT^n)^I$ is induced from the tautological bundle over $BT^1$ if $I_i = 0$ 
and the zero bundle if $I_i = 1$.
Thus, equation (\ref{eqn:roots}) is true.
\end{proof}

We can use this to calculate the Waner classes in the cohomology of $BU(2)$.

\begin{proposition}
In the cohomology of $BU(2)$, we have
\begin{align*}
    c_{\omega-2} &= \zeta_2^2 \cl \\
\intertext{and}
    c_{\omega-4} &= \zeta_1\zeta_2^2 = \zeta^{\omega-4}.
\end{align*}
\end{proposition}

\begin{proof}
In the cohomology of $BT^2$, the Whitney sum formula tells us that
\begin{align*}
    s^*(W(\omega))
    &= W(\omega_1\dirsum\omega_2) \\
    &= (\zeta^{\omega_1-2} + \clo t)(\zeta^{\omega_2-2} + \clt t) \\
    &= \zeta^{\omega_1+\omega_2-4} + (\zeta^{\omega_2-2}\clo + \zeta^{\omega_1-2}\clt)t + \clo\clt t^2 \\
    &= s^*(\zeta^{\omega-4}) + s^*(\zeta_2^2\cl)t + s^*(\cw)t^2
\end{align*}
For the middle term, the computation is
\begin{align*}
    s^*(\zeta_2^2\cl) &= \zeta_{11}^2\cltens \\
        &= \zeta_{11}(\zeta_{01}\clo + \zeta_{10}\clt - e^{-2}\kappa \zeta_{00}\zeta_{01}\zeta_{10}\clo\clt) \\
        &= \zeta^{\omega_2-2}\clo + \zeta^{\omega_1-2}\clt.
\end{align*}
(The third term vanishes because $e^{-2}\kappa\cdot\xi = 0$.)
Since we are in gradings in which $s^*$ is injective, 
this gives
\[
    W(\omega) = \zeta^{\omega-4} + \zeta_2^2\cl t + \cw t^2,
\]
verifying the claims of the proposition.
\end{proof}

Thus, the Waner classes $c_{\omega-2}$ and $c_{\omega-4}$ are basis elements, but not multiplicative generators,
hence not as fundamental as, for example, the class $\cl$.

\section*{Appendix}

We give here the verification that we can resolve all ambiguities in the reduction system
used to prove Proposition~\ref{prop:freeness}. 
For each pair of reductions, with left sides $W_1$ and $W_2$, it suffices to look at
the least common multiple of $W_1$ and $W_2$, which can be reduced in two different ways using
the two reductions.
We need to show that a series of further reductions can be
applied to bring the resulting polynomials to the same place.
One simplification we can make: If $W_1$ and $W_2$ have greatest common divisor 1, hence least common multiple $W_1 W_2$,
then the verification is trivial: Applying one reduction and then the other in either order leads
to the same polynomial. Thus, we can take those pairings as verified. We now list all the pairings
together with the resolutions of their ambiguities.
The numbering of the reductions is as in the proof of Proposition~\ref{prop:freeness}.

\begin{itemize}
\item \ref{red:1} and \ref{red:2}: The greatest common divisor is 1.

\item \ref{red:1} and \ref{red:3}: The least common multiple is $\zeta_{00}\zeta_{01}\zeta_{11}\cxlt$.
\begin{align*}
	\zeta_{00}\zeta_{01}\zeta_{11}\cxlt
		&\xmapsto{\ref{red:1}} \zeta^{\chi\omega_1-2}\zeta_{11}\cxlt \\
	\zeta_{00}\zeta_{01}\zeta_{11}\cxlt
		&\xmapsto{\ref{red:3}} (1-\kappa)\zeta_{00}^2\zeta_{10}\clt + e^2\zeta_{00} \\
		&\xmapsto{\ref{red:4}} (1-\kappa)\zeta^{\chi\omega_1-2}\zeta_{10}\cxltens \\
		&\qquad - (1 - \epsilon_1)\cxlo\zeta_{00}\zeta_{10}\zeta_{11}
			+ e^2\zeta_{00} \\
		&\xmapsto{\ref{red:11}} (1-\kappa)\zeta^{\chi\omega_1-2}\zeta_{00}\clo \\
		&\qquad + (1-\kappa)(1-\epsilon_1)\zeta^{\chi\omega_1-2}\zeta_{11}\clt \\
		&\qquad - (1 - \epsilon_1)\cxlo\zeta_{00}\zeta_{10}\zeta_{11}
			+ e^2\zeta_{00} \\
		&\xmapsto{\ref{red:2}} (1-\kappa)\zeta^{\chi\omega_1-2}\zeta_{00}\clo \\
		&\qquad + (1-\kappa)(1-\epsilon_1)\zeta^{\chi\omega_1-2}\zeta_{11}\clt \\
		&\qquad - (1 - \epsilon_1)\zeta^{\omega_1-2}\cxlo\zeta_{00}
			+ e^2\zeta_{00} \\
		&= (1-\kappa)\zeta^{\chi\omega_1-2}\zeta_{00}\clo + \zeta^{\chi\omega_1-2}\zeta_{11}\clt \\
		&\qquad - \zeta^{\omega_1-2}\cxlo\zeta_{00} + e^2\zeta_{00} \\
		&= \zeta^{\chi\omega_1-2}\zeta_{11}\clt
\end{align*}
The last two equalities come from equalities in $ H_\GG^{RO(\Pi U_1)}((U_1)_+)$.

\item \ref{red:1} and \ref{red:4}
\begin{align*}
	\zeta_{00}^2\zeta_{01}\clt
		&\xmapsto{\ref{red:1}} \zeta^{\chi\omega_1-2}\zeta_{00}\clt \\
	\zeta_{00}^2\zeta_{01}\clt
		&\xmapsto{\ref{red:4}} \zeta^{\chi\omega_1-2}\zeta_{01}\cxltens \\
		&\qquad - (1 - \kappa)(1 - \epsilon_1)\cxlo\zeta_{00}\zeta_{01}\zeta_{11} \\
		&\xmapsto{\ref{red:1}} \zeta^{\chi\omega_1-2}\zeta_{01}\cxltens \\
		&\qquad - (1 - \kappa)(1 - \epsilon_1)\zeta^{\chi\omega_1-2}\cxlo\zeta_{11} \\
		&= \zeta^{\chi\omega_1-2}\zeta_{01}\cxltens - \zeta^{\chi\omega_1-2}\cxlo\zeta_{11} \\
		&\xmapsto{\ref{red:10}} (1-\kappa)(1-\epsilon_1)\zeta^{\chi\omega_1-2}\zeta_{00}\clt \\
		&= \zeta^{\chi\omega_1-2}\zeta_{00}\clt
\end{align*}

\item \ref{red:1} and \ref{red:5}: The greatest common divisor is 1.

\item \ref{red:1} and \ref{red:6}: Similar to \ref{red:1} and \ref{red:4}.

\item \ref{red:1} and \ref{red:7}: The greatest common divisor is 1.

\item \ref{red:1} and \ref{red:8}
\begin{align*}
	\zeta_{00}\zeta_{01}\cltens
		&\xmapsto{\ref{red:1}} \zeta^{\chi\omega_1-2} \cltens \\
	\zeta_{00}\zeta_{01}\cltens
		&\xmapsto{\ref{red:8}} \cxlo\zeta_{01}\zeta_{10} \\
		&\qquad	+ (1-\kappa)(1-\epsilon_1)\zeta_{01}^2 \cxlt \\
		&\xmapsto{\ref{red:6}} (1-\kappa)(1-\epsilon_1)\zeta^{\chi\omega_1-2}\cltens \\
		&= \zeta^{\chi\omega_1-2}\cltens
\end{align*}

\item \ref{red:1} and \ref{red:9}: The greatest common divisor is 1.

\item \ref{red:1} and \ref{red:10}: Similar to \ref{red:1} and \ref{red:8}.

\item \ref{red:1} and \ref{red:11}: The greatest common divisor is 1.

\item \ref{red:1} and \ref{red:12}: The greatest common divisor is 1.

\item \ref{red:2} and \ref{red:3}
\begin{align*}
	\zeta_{01}\zeta_{10}\zeta_{11} \cxlt
		&\xmapsto{\ref{red:2}} \zeta^{\omega_1-2}\zeta_{01} \cxlt \\
	\zeta_{01}\zeta_{10}\zeta_{11} \cxlt
		&\xmapsto{\ref{red:3}} (1-\kappa)\zeta_{00}\zeta_{10}^2\clt + e^2\zeta_{10} \\
		&\xmapsto{\ref{red:5}} (1-\kappa)\zeta^{\omega_1-2}\zeta_{00}\cltens \\
		&\qquad - (1-\kappa)(1 - \epsilon_1)\clo\zeta_{00}\zeta_{01}\zeta_{10} \\
		&\qquad + e^2\zeta_{10} \\
		&\xmapsto{\ref{red:1}} (1-\kappa)\zeta^{\omega_1-2}\zeta_{00}\cltens \\
		&\qquad - (1-\kappa)(1 - \epsilon_1)\zeta^{\chi\omega_1-2}\clo\zeta_{10} \\
		&\qquad + e^2\zeta_{10} \\
		&= (1-\kappa)\zeta^{\omega_1-2}\zeta_{00}\cltens \\
		&\qquad - \zeta^{\chi\omega_1-2} \clo\zeta_{10} + e^2\zeta_{10} \\
		&\xmapsto{\ref{red:8}} (1-\kappa)\zeta^{\omega_1-2}\cxlo\zeta_{10} \\
		&\qquad + (1-\epsilon_1)\zeta^{\omega_1-2}\zeta_{01}\cxlt \\
		&\qquad - \zeta^{\chi\omega_1-2} \clo\zeta_{10} + e^2\zeta_{10} \\
		&= \zeta^{\omega_1-2}\zeta_{01}\cxlt
\end{align*}

\item \ref{red:2} and \ref{red:4}: The greatest common divisor is 1.

\item \ref{red:2} and \ref{red:5}: Similar to \ref{red:1} and \ref{red:4}.

\item \ref{red:2} and \ref{red:6}: The greatest common divisor is 1.

\item \ref{red:2} and \ref{red:7}: Similar to \ref{red:1} and \ref{red:4}.

\item \ref{red:2} and \ref{red:8}: The greatest common divisor is 1.

\item \ref{red:2} and \ref{red:9}: Similar to \ref{red:1} and \ref{red:8}.

\item \ref{red:2} and \ref{red:10}: The greatest common divisor is 1.

\item \ref{red:2} and \ref{red:11}: Similar to \ref{red:1} and \ref{red:8}.

\item \ref{red:2} and \ref{red:12}: The greatest common divisor is 1.

\item \ref{red:3} and \ref{red:4}: The greatest common divisor is 1.

\item \ref{red:3} and \ref{red:5}: The greatest common divisor is 1.

\item \ref{red:3} and \ref{red:6}
\begin{align*}
	\zeta_{01}^2\zeta_{11}\cxlt
		&\xmapsto{\ref{red:3}} (1-\kappa)\zeta_{00}\zeta_{01}\zeta_{10}\clt + e^2\zeta_{01} \\
		&\xmapsto{\ref{red:1}} (1-\kappa)\zeta^{\chi\omega_1-2}\zeta_{10}\clt + e^2\zeta_{01} \\
	\zeta_{01}^2\zeta_{11}\cxlt
		&\xmapsto{\ref{red:6}} \zeta^{\chi\omega_1-2}\zeta_{11}\cltens \\
		&\qquad - (1 - \kappa)(1 - \epsilon_1)\cxlo\zeta_{01}\zeta_{10}\zeta_{11} \\
		&\xmapsto{\ref{red:2}} \zeta^{\chi\omega_1-2}\zeta_{11}\cltens \\
		&\qquad - (1 - \kappa)(1 - \epsilon_1)\zeta^{\omega_1-2}\cxlo\zeta_{01} \\
		&= \zeta^{\chi\omega_1-2}\zeta_{11}\cltens - (1 - \kappa)\zeta^{\omega_1-2}\cxlo\zeta_{01} \\
		&\xmapsto{\ref{red:9}} \zeta^{\chi\omega_1-2}\clo \zeta_{01}
				+ (1-\epsilon_1)\zeta^{\chi\omega_1-2}\zeta_{10}\clt \\
		&\qquad - (1 - \kappa)\zeta^{\omega_1-2}\cxlo\zeta_{01} \\
		&= (1-\kappa)\zeta^{\chi\omega_1-2}\zeta_{10}\clt + e^2\zeta_{01}
\end{align*}

\item \ref{red:3} and \ref{red:7}: Similar to \ref{red:3} and \ref{red:6}.

\item \ref{red:3} and \ref{red:8}: The greatest common divisor is 1.

\item \ref{red:3} and \ref{red:9}
\begin{align*}
	\zeta_{01}\zeta_{11}\cxlt\cltens
		&\xmapsto{\ref{red:3}} (1-\kappa)\zeta_{00}\zeta_{10}\clt\cltens 
			+ e^2 \cltens \\
		&\xmapsto{\ref{red:8}} (1-\kappa)\cxlo\zeta_{10}^2 \clt \\
		&\qquad + (1-\epsilon_1)\zeta_{01}\zeta_{10} \clt\cxlt \\
		&\qquad + e^2 \cltens \\
		&\xmapsto{\ref{red:5}} (1-\kappa)\zeta^{\omega_1-2}\cxlo\cltens \\
		&\qquad - (1-\kappa)(1 - \epsilon_1)\clo\cxlo\zeta_{01}\zeta_{10} \\
		&\qquad + (1-\epsilon_1)\zeta_{01}\zeta_{10} \clt\cxlt \\
		&\qquad + e^2 \cltens \\
		&= \zeta^{\chi\omega_1-2}\clo\cltens \\
		&\qquad - (1-\kappa)(1 - \epsilon_1)\clo\cxlo\zeta_{01}\zeta_{10} \\
		&\qquad + (1-\epsilon_1)\zeta_{01}\zeta_{10} \clt\cxlt \\
	\zeta_{01}\zeta_{11}\cxlt\cltens
		&\xmapsto{\ref{red:9}} \clo\zeta_{01}^2 \cxlt \\
		&\qquad + (1-\epsilon_1)\zeta_{01}\zeta_{10}\clt\cxlt \\
		&\xmapsto{\ref{red:6}} \zeta^{\chi\omega_1-2}\clo\cltens \\
		&\qquad - (1 - \kappa)(1 - \epsilon_1)\clo\cxlo\zeta_{01}\zeta_{10} \\
		&\qquad + (1-\epsilon_1)\zeta_{01}\zeta_{10}\clt\cxlt \\
\end{align*}

\item \ref{red:3} and \ref{red:10}: Similar to \ref{red:3} and \ref{red:9}.

\item \ref{red:3} and \ref{red:11}: The greatest common divisor is 1.

\item \ref{red:3} and \ref{red:12}: The greatest common divisor is 1.

\item \ref{red:4} and \ref{red:5}
\begin{align*}
	\zeta_{00}^2\zeta_{10}^2 \clt
		&\xmapsto{\ref{red:4}} \zeta^{\chi\omega_1-2}\zeta_{10}^2\cxltens \\
		&\qquad - (1 - \kappa)(1 - \epsilon_1)\cxlo\zeta_{00}\zeta_{10}^2\zeta_{11} \\
		&\xmapsto{\ref{red:2}} \zeta^{\chi\omega_1-2}\zeta_{10}^2\cxltens \\
		&\qquad - (1 - \kappa)(1 - \epsilon_1)\zeta^{\omega_1-2}\cxlo\zeta_{00}\zeta_{10} \\
		&\xmapsto{\ref{red:11}} \zeta^{\chi\omega_1-2}\clo\zeta_{00}\zeta_{10}  \\
		&\qquad + (1-\epsilon_1)\zeta^{\chi\omega_1-2}\zeta_{10}\zeta_{11}\cxlt \\
		&\qquad - (1 - \kappa)(1 - \epsilon_1)\zeta^{\omega_1-2}\cxlo\zeta_{00}\zeta_{10} \\
		&= \zeta^{\chi\omega_1-2}\clo\zeta_{00}\zeta_{10}  \\
		&\qquad + (1-\kappa)\zeta^{\chi\omega_1-2}\zeta_{10}\zeta_{11}\cxlt \\
		&\qquad - (1 - \kappa)\zeta^{\omega_1-2}\cxlo\zeta_{00}\zeta_{10} \\
		&= (1-\kappa)\zeta^{\chi\omega_1-2}\zeta_{10}\zeta_{11}\cxlt + e^2\zeta_{00}\zeta_{10} \\
		&\xmapsto{\ref{red:2}} (1-\kappa)\zeta^{\chi\omega_1-2}\zeta^{\omega_1-2}\cxlt + e^2\zeta_{00}\zeta_{10} \\
		&= \xi \cxlt + e^2\zeta_{00}\zeta_{10} \\
	\zeta_{00}^2\zeta_{10}^2 \clt
		&\xmapsto{\ref{red:5}} \zeta^{\omega_1-2}\zeta_{00}^2\cltens \\
		&\qquad - (1 - \epsilon_1)\clo\zeta_{00}^2\zeta_{01}\zeta_{10} \\
		&\xmapsto{\ref{red:1}} \zeta^{\omega_1-2}\zeta_{00}^2\cltens \\
		&\qquad - (1 - \epsilon_1)\zeta^{\chi\omega_1-2}\clo\zeta_{00}\zeta_{10} \\
		&\xmapsto{\ref{red:8}} \zeta^{\omega_1-2}\cxlo \zeta_{00}\zeta_{10} \\
		&\qquad + (1-\kappa)(1-\epsilon_1)\zeta^{\omega_1-2}\zeta_{00}\zeta_{01}\cxlt \\
		&\qquad - (1 - \epsilon_1)\zeta^{\chi\omega_1-2}\clo\zeta_{00}\zeta_{10} \\
		&= \zeta^{\omega_1-2}\cxlo \zeta_{00}\zeta_{10} \\
		&\qquad + (1-\kappa)\zeta^{\omega_1-2}\zeta_{00}\zeta_{01}\cxlt \\
		&\qquad - (1 - \kappa)\zeta^{\chi\omega_1-2}\clo\zeta_{00}\zeta_{10} \\
		&= (1-\kappa)\zeta^{\omega_1-2}\zeta_{00}\zeta_{01}\cxlt + e^2\zeta_{00}\zeta_{10} \\
		&\xmapsto{\ref{red:1}} (1-\kappa)\zeta^{\omega_1-2}\zeta^{\chi\omega_1-2}\cxlt + e^2\zeta_{00}\zeta_{10} \\
		&= \xi \cxlt + e^2\zeta_{00}\zeta_{10}
\end{align*}

\item \ref{red:4} and \ref{red:6}: The greatest common divisor is 1.

\item \ref{red:4} and \ref{red:7}: The greatest common divisor is 1.

\item \ref{red:4} and \ref{red:8}
\begin{align*}
	\zeta_{00}^2 \clt\cltens
		&\xmapsto{\ref{red:4}} \zeta^{\chi\omega_1-2}\cltens\cxltens \\
		&\qquad - (1 - \kappa)(1 - \epsilon_1)\cxlo\zeta_{00}\zeta_{11}\cltens \\
		&\xmapsto{\ref{red:12}} \zeta^{\chi\omega_1-2}\clo\cxlo + \zeta^{\chi\omega_1-2}\clt\cxlt \\
		&\qquad + \tau(\iota^{-2})\zeta^{2(\chi\omega_1-2)}\clo\zeta_{00}\zeta_{10}\clt \\
		&\qquad - (1 - \kappa)(1 - \epsilon_1)\cxlo\zeta_{00}\zeta_{11}\cltens \\
		&\xmapsto{\ref{red:9}} \zeta^{\chi\omega_1-2}\clo\cxlo + \zeta^{\chi\omega_1-2}\clt\cxlt \\
		&\qquad + \tau(\iota^{-2})\zeta^{2(\chi\omega_1-2)}\clo\zeta_{00}\zeta_{10}\clt \\
		&\qquad - (1 - \kappa)(1 - \epsilon_1)\clo\cxlo\zeta_{00}\zeta_{01} \\
		&\qquad - (1-\kappa)\cxlo\zeta_{00}\zeta_{10}\clt \\
		&\xmapsto{\ref{red:1}} \zeta^{\chi\omega_1-2}\clo\cxlo + \zeta^{\chi\omega_1-2}\clt\cxlt \\
		&\qquad + \tau(\iota^{-2})\zeta^{2(\chi\omega_1-2)}\clo\zeta_{00}\zeta_{10}\clt \\
		&\qquad - (1 - \kappa)(1 - \epsilon_1)\zeta^{\chi\omega_1-2}\clo\cxlo \\
		&\qquad - (1-\kappa)\cxlo\zeta_{00}\zeta_{10}\clt \\
		&= \zeta^{\chi\omega_1-2}\clt\cxlt 
			+ \cxlo\zeta_{00}\zeta_{10}\clt \\
	\zeta_{00}^2 \clt\cltens
		&\xmapsto{\ref{red:8}} \cxlo\zeta_{00}\zeta_{10}\clt \\
		&\qquad + (1-\kappa)(1-\epsilon_1)\zeta_{00}\zeta_{01}\clt\cxlt \\
		&\xmapsto{\ref{red:1}}  \cxlo\zeta_{00}\zeta_{10}\clt \\
		&\qquad + (1-\kappa)(1-\epsilon_1)\zeta^{\chi\omega_1-2}\clt\cxlt \\
		&= \cxlo\zeta_{00}\zeta_{10}\clt + \zeta^{\chi\omega_1-2}\clt\cxlt
\end{align*}

\item \ref{red:4} and \ref{red:9}: The greatest common divisor is 1.

\item \ref{red:4} and \ref{red:10}: The greatest common divisor is 1.

\item \ref{red:4} and \ref{red:11}: The greatest common divisor is 1.

\item \ref{red:4} and \ref{red:12}: The greatest common divisor is 1.

\item \ref{red:5} and \ref{red:6}: The greatest common divisor is 1.

\item \ref{red:5} and \ref{red:7}: The greatest common divisor is 1.

\item \ref{red:5} and \ref{red:8}: The greatest common divisor is 1.

\item \ref{red:5} and \ref{red:9}: The greatest common divisor is 1.

\item \ref{red:5} and \ref{red:10}: The greatest common divisor is 1.

\item \ref{red:5} and \ref{red:11}: Similar to \ref{red:4} and \ref{red:8}.

\item \ref{red:5} and \ref{red:12}: The greatest common divisor is 1.

\item \ref{red:6} and \ref{red:7}: Similar to \ref{red:4} and \ref{red:5}.

\item \ref{red:6} and \ref{red:8}: The greatest common divisor is 1.

\item \ref{red:6} and \ref{red:9}: The greatest common divisor is 1.

\item \ref{red:6} and \ref{red:10}: Similar to \ref{red:4} and \ref{red:8}.

\item \ref{red:6} and \ref{red:11}: The greatest common divisor is 1.

\item \ref{red:6} and \ref{red:12}: The greatest common divisor is 1.

\item \ref{red:7} and \ref{red:8}: The greatest common divisor is 1.

\item \ref{red:7} and \ref{red:9}: Similar to \ref{red:4} and \ref{red:8}.

\item \ref{red:7} and \ref{red:10}: The greatest common divisor is 1.

\item \ref{red:7} and \ref{red:11}: The greatest common divisor is 1.

\item \ref{red:7} and \ref{red:12}: The greatest common divisor is 1.

\item \ref{red:8} and \ref{red:9}
\begin{align*}
	\zeta_{00}\zeta_{11} \cltens
		&\xmapsto{\ref{red:8}} \zeta_{10}\zeta_{11}\cxlo \\
		&\qquad + (1-\kappa)(1-\epsilon_1)\zeta_{01}\zeta_{11}\cxlt \\
		&\xmapsto{\ref{red:2}} \zeta^{\omega_1-2}\cxlo \\
		&\qquad + (1-\kappa)(1-\epsilon_1)\zeta_{01}\zeta_{11}\cxlt \\
		&\xmapsto{\ref{red:3}}  \zeta^{\omega_1-2}\cxlo \\
		&\qquad + (1-\epsilon_1)\zeta_{00}\zeta_{10}\clt \\
		&\qquad - e^2 + \kappa\zeta^{\chi\omega_1-2}\clo \\
		&= \zeta^{\chi\omega_1-2}\clo \\
		&\qquad + (1-\epsilon_1)\zeta_{00}\zeta_{10}\clt \\
	\zeta_{00}\zeta_{11} \cltens
		&\xmapsto{\ref{red:9}} \zeta_{00}\zeta_{01}\clo \\
		&\qquad + (1-\epsilon_1)\zeta_{00}\zeta_{10}\clt \\
		&\xmapsto{\ref{red:1}} \zeta^{\chi\omega_1-2}\clo \\
		&\qquad + (1-\epsilon_1)\zeta_{00}\zeta_{10}\clt \\
\end{align*}

\item \ref{red:8} and \ref{red:10}: The greatest common divisor is 1.

\item \ref{red:8} and \ref{red:11}: The greatest common divisor is 1.

\item \ref{red:8} and \ref{red:12}
\begin{align*}
	\zeta_{00}\cltens\cxltens
		&\xmapsto{\ref{red:8}} \cxlo\zeta_{10}\cxltens \\
		&\qquad + (1-\kappa)(1-\epsilon_1)\zeta_{01}\cxlt\cxltens \\
		&\xmapsto{\ref{red:11}} \clo\cxlo\zeta_{00} \\
		&\qquad + (1-\epsilon_1)\cxlo\zeta_{11}\cxlt \\
		&\qquad + (1-\kappa)(1-\epsilon_1)\zeta_{01}\cxlt\cxltens \\
		&\xmapsto{\ref{red:10}} \clo\cxlo\zeta_{00} \\
		&\qquad + (2-\kappa)(1-\epsilon_1)\cxlo\zeta_{11}\cxlt \\
		&\qquad + \zeta_{00}\clt\cxlt \\
	\zeta_{00}\cltens\cxltens
		&\xmapsto{\ref{red:12}} \clo\cxlo\zeta_{00} + \zeta_{00}\clt\cxlt \\
		&\qquad + \tau(\iota^{-2})\zeta^{\chi\omega_1-2}\clo\zeta_{00}^2\zeta_{10}\clt \\
		&\xmapsto{\ref{red:4}} \clo\cxlo\zeta_{00} + \zeta_{00}\clt\cxlt \\
		&\qquad + \tau(\iota^{-2})\zeta^{2(\chi\omega_1-2)}\clo\zeta_{10}\cxltens \\
		&\qquad - \tau(\iota^{-2})\zeta^{\chi\omega_1-2}\clo\cxlo\zeta_{00}\zeta_{10}\zeta_{11} \\
		&\xmapsto{\ref{red:2}} \clo\cxlo\zeta_{00} + \zeta_{00}\clt\cxlt \\
		&\qquad + \tau(\iota^{-2})\zeta^{2(\chi\omega_1-2)}\clo\zeta_{10}\cxltens \\
		&\qquad - \tau(\iota^{-2})\zeta^{\chi\omega_1-2}\zeta^{\omega_1-2}\clo\cxlo\zeta_{00} \\
		&= -(1-\kappa)\clo\cxlo\zeta_{00} + \zeta_{00}\clt\cxlt \\
		&\qquad + (2-\kappa)\cxlo\zeta_{10}\cxltens \\
		&\xmapsto{\ref{red:11}} \clo\cxlo\zeta_{00} + \zeta_{00}\clt\cxlt \\
		&\qquad + (2-\kappa)(1-\epsilon_1)\cxlo\zeta_{11}\cxlt \\
\end{align*}

\item \ref{red:9} and \ref{red:10}: The greatest common divisor is 1.

\item \ref{red:9} and \ref{red:11}: The greatest common divisor is 1.

\item \ref{red:9} and \ref{red:12}: Similar to \ref{red:8} and \ref{red:12}.

\item \ref{red:10} and \ref{red:11}: Similar to \ref{red:8} and \ref{red:9}.

\item \ref{red:10} and \ref{red:12}: Similar to \ref{red:8} and \ref{red:12}.

\item \ref{red:11} and \ref{red:12}: Similar to \ref{red:8} and \ref{red:12}.

\end{itemize}

\bibliographystyle{amsplain}
\bibliography{Bibliography}

\providecommand{\bysame}{\leavevmode\hbox to3em{\hrulefill}\thinspace}
\providecommand{\MR}{\relax\ifhmode\unskip\space\fi MR }
\providecommand{\MRhref}[2]{%
  \href{http://www.ams.org/mathscinet-getitem?mr=#1}{#2}
}
\providecommand{\href}[2]{#2}
\begin{thebibliography}{10}

\bibitem{Beaudry:Guide}
Agnès Beaudry, Chloe Lewis, Clover May, Sabrina Pauli, and Elizabeth Tatum,
  \emph{A guide to equivariant parametrized cohomology}, preprint,
  \url{https://arxiv.org/abs/2410.13971}, 2024.

\bibitem{Berg:diamondLemma}
G.~M. Bergman, \emph{The diamond lemma for ring theory}, Adv. in Math.
  \textbf{29} (1978), no.~2, 178--218. \MR{506890}

\bibitem{CGK:formalGroupLaws}
Michael Cole, J.~P.~C. Greenlees, and I.~Kriz, \emph{Equivariant formal group
  laws}, Proc. London Math. Soc. (3) \textbf{81} (2000), no.~2, 355--386.
  \MR{1770613}

\bibitem{Co:InfinitePublished}
S.~R. Costenoble, \emph{The {$RO(\Pi B)$}-graded {$C_2$}-equivariant ordinary
  cohomology of {$B_{C_2} U(1)$}}, Topology Appl. \textbf{338} (2023), Paper
  No. 108660, 53pp.

\bibitem{CH:bu2}
S.~R. Costenoble and T.~Hudson, \emph{The {$\GG$}-equivariant ordinary
  cohomology of {$B U(2)$}}, preprint, 2024.

\bibitem{CHTFiniteProjSpace}
S.~R. Costenoble, T.~Hudson, and S.~Tilson, \emph{The {$C_2$}-equivariant
  cohomology of complex projective spaces}, Adv. Math. \textbf{398} (2022),
  Paper No. 108245, 69pp. \MR{4388953}

\bibitem{CHTAlgebraic}
\bysame, \emph{An algebraic {$C_2$}-equivariant {B}\'ezout theorem}, Algebr.
  Geom. Topol. \textbf{24} (2024), no.~4, 2331--2350. \MR{4776283}

\bibitem{CostenobleWanerBook}
S.~R. Costenoble and S.~Waner, \emph{Equivariant ordinary homology and
  cohomology}, Lecture Notes in Mathematics, vol. 2178, Springer, Cham, 2016.
  \MR{3585352}

\bibitem{DuggerGrass}
D.~Dugger, \emph{Bigraded cohomology of {$\mathbb{Z}/2$}-equivariant
  {G}rassmannians}, Geom. Topol. \textbf{19} (2015), no.~1, 113--170.
  \MR{3318749}

\bibitem{HazelFundamental}
C.~Hazel, \emph{Equivariant fundamental classes in {${\rm RO}(C_2)$}-graded
  cohomology with {$\Bbb Z/2$}-coefficients}, Algebr. Geom. Topol. \textbf{21}
  (2021), no.~6, 2799--2856. \MR{4344871}

\bibitem{HazelSurfaces}
\bysame, \emph{The cohomology of {$C_2$}-surfaces with {$\underline{\Bbb
  Z}$}-coefficients}, J. Homotopy Relat. Struct. \textbf{18} (2023), no.~1,
  71--114. \MR{4555910}

\bibitem{Hogle}
E.~Hogle, \emph{{$RO(C_2)$}-graded cohomology of equivariant {G}rassmannian
  manifolds}, New York J. Math. \textbf{27} (2021), 53--98. \MR{4195417}

\bibitem{KronholmSerre}
W.~Kronholm, \emph{The {${\rm RO}(G)$}-graded {S}erre spectral sequence},
  Homology Homotopy Appl. \textbf{12} (2010), no.~1, 75--92. \MR{2607411}

\bibitem{LewisCP}
L.~Gaunce Lewis, Jr., \emph{The {$R{\rm O}(G)$}-graded equivariant ordinary
  cohomology of complex projective spaces with linear {${\bf Z}/p$} actions},
  Algebraic topology and transformation groups ({G}\"{o}ttingen, 1987), Lecture
  Notes in Math., vol. 1361, Springer, Berlin, 1988, pp.~53--122. \MR{979507}

\bibitem{Wan:ChernClasses}
S.~Waner, \emph{Equivariant {C}hern classes}, unpublished manuscript, 1983.

\end{thebibliography}

\end{document}